\documentclass[a4paper,twoside,fleqn,10pt]{article}

\usepackage{graphicx}
\usepackage{latexsym}
\usepackage{amsmath}
\usepackage{amssymb}
\usepackage{amsthm}
\usepackage{booktabs}

\newcommand{\version}{2009/nov/6}
\newcommand{\shorttitle}{Stabdom A}
\newcommand{\shortauthor}{Hvn, Krv}

\title{Singularities on the boundary of the stability domain\\near 1:1 resonance}
\author{I. Hoveijn\\
\small Department of Mathematics, University of Groningen\\
\small P.O. Box 800, 9700 AV Groningen, The Netherlands\\\\
O.N. Kirillov\\
\small Dynamics and Vibrations Group\\
\small Department of Mechanical Engineering, TU-Darmstadt\\
\small Hochschulstrasse 1, 64289 Darmstadt, Germany}
\date{\version}

\newcommand{\abc}{
\renewcommand{\theenumi}{\alph{enumi}}
\renewcommand{\labelenumi}{\theenumi)}
\itemsep 0pt
}

\newtheorem{theorem}{Theorem}[section]

\newtheorem{corollary}[theorem]{Corollary}
\newtheorem{definition}{Definition}[section]
\newtheorem{lemma}[theorem]{Lemma}
\newtheorem{proposition}[theorem]{Proposition}
\newtheorem{Remark}{Remark}[section]
\newtheorem{Example}[Remark]{Example}

\newenvironment{remark}{\begin{Remark} \begin{rm}}{\hfill$\blacktriangleright$\end{rm} \end{Remark}}

\newcommand{\Ad}{\textrm{Ad}}
\newcommand{\Dm}{D^-}
\newcommand{\Dpm}{D^{\pm}}
\newcommand{\Dp}{D^+}
\newcommand{\Gl}{\mathbf{Gl}}
\newcommand{\G}{\mathbf{G}}
\newcommand{\Hess}{\textrm{Hess}}
\newcommand{\M}{\mathcal{M}}
\newcommand{\Mzza}{\mathcal{M}_{678}}
\newcommand{\Orb}{\textrm{Orb}}
\newcommand{\R}{\mathcal{R}}
\newcommand{\Sgn}{\textrm{Sgn}}
\newcommand{\SO}{\mathbf{SO}}
\newcommand{\Sm}{S^-}
\newcommand{\Spm}{S^{\pm}}
\newcommand{\Sp}{S^+}
\newcommand{\basis}[1]{\langle #1 \rangle}

\newcommand{\breukje}[2]{\mbox{\small $\frac{#1}{#2}$}}
\newcommand{\cl}[1]{\overline{#1}}
\newcommand{\commentaar}[1]{{}}
\newcommand{\dstabdom}{\partial\Sigma}
\newcommand{\e}{\textrm{e}}

\newcommand{\fR}{\mathbb{R}}
\newcommand{\fZ}{\mathbb{Z}}
\newcommand{\gl}{\mathbf{gl}}
\newcommand{\g}{\mathbf{g}}
\newcommand{\halfje}{\breukje{1}{2}}
\newcommand{\hpm}{\hphantom{-}}
\newcommand{\hwt}{\halfje\sqrt{2}}
\newcommand{\id}{\textrm{id}}
\newcommand{\inprod}[2]{\langle #1, #2 \rangle}
\newcommand{\m}{\mathbf{m}}
\newcommand{\mzza}{\mathbf{m}_{678}}

\newcommand{\phim}{\phi^-}
\newcommand{\phipm}{\phi^{\pm}}
\newcommand{\phip}{\phi^+}
\newcommand{\so}{\mathbf{so}}
\newcommand{\stabdom}{\Sigma}
\newcommand{\trace}{\textrm{Trace}}

\renewcommand{\H}{\mathcal{H}}
\renewcommand{\L}{\mathcal{L}}
\renewcommand{\Re}{\text{Re}}
\renewcommand{\Im}{\text{Im}}

\commentaar{ to check or to complete
+ thread ? (no thread, eigenvalues split at $\beta^2$)
+ $Gl_L(4,\fR)$ and $\M$ (!), invariant slice theorem, zie ook ttd
+ MCC
+ codims
+ universality
+ Holtman
+ Pluecker
+ ref Kuz
+ refs HH en HSN
+ $0^2\beta$ symplectic ?? (neen)
+ def and ref of Whitney stratification (springer encyclopaedia)
+ caption of figure transorb
+ caption table eigconfs
+ tentative title of B
+ sings in 3 param families, naming
+ refs, keep refs plu, cat, doc ?
+ phi en psi, identifications, caption figure
+ coms en opms
+ spelling
+ abstract
}

\begin{document}
\maketitle


\begin{abstract}\noindent
We study the linear differential equation $\dot{x} = Lx$ in $1$:$1$ resonance. That is, $x \in \fR^4$ and $L$ is $4 \times 4$ matrix with a semi-simple double pair of imaginary eigenvalues $(i\beta,-i\beta,i\beta,-i\beta)$. We wish to find all perturbations of this linear system such that the perturbed system is stable. Since linear differential equations are in one to one correspondence with linear maps we translate this problem to $\gl(4,\fR)$. In this setting our aim is to determine the stability domain and the singularities of its boundary. The dimension of $\gl(4,\fR)$ is $16$, therefore we first reduce the dimension as far as possible. Here we use a versal unfolding of $L$ ie a transverse section of the orbit of $L$ under the adjoint action of $\Gl(4,\fR)$. Repeating a similar procedure in the versal unfolding we are able to reduce the dimension to $4$. A $3$-sphere in this $4$-dimensional space contains all information about the neighborhood of $L$ in $\gl(4,\fR)$. Considering the $3$-sphere as two $3$-discs glued smoothly along their common boundary we find that the boundary of the stability domain is contained in two right conoids, one in each $3$-disc. The singularities of this surface are transverse self-intersections, Whitney umbrellas and an intersection of self-intersections where the surface has a self-tangency. A Whitney stratification of the $3$-sphere such that the eigenvalue configurations of corresponding matrices are constant on strata allows us to describe the neighborhood of $L$ and in particular identify the stability domain.
\end{abstract}

\commentaar{We study the linear differential equation x' = Lx in 1:1 resonance. That is, x R^4 and L is 4 by 4 matrix with a semi-simple double pair of imaginary eigenvalues (ib,-ib,ib,-ib). We wish to find all perturbations of this linear system such that the perturbed system is stable. Since linear differential equations are in one to one correspondence with linear maps we translate this problem to gl(4,R). In this setting our aim is to determine the stability domain and the singularities of its boundary. The dimension of gl(4,R) is 16, therefore we first reduce the dimension as far as possible. Here we use a versal unfolding of L ie a transverse section of the orbit of L under the adjoint action of Gl(4,R). Repeating a similar procedure in the versal unfolding we are able to reduce the dimension to 4. A 3-sphere in this 4-dimensional space contains all information about the neighborhood of L in gl(4,R). In this 3-sphere the boundary of the stability domain is a surface with singularities: transverse self-intersections, Whitney umbrellas and an intersection of self-intersections where the surface has a self-tangency.}

%
\section{Introduction}\label{sec:intro}

\subsection{Setting}\label{sec:introset}
Suppose that the ordinary differential equation
\begin{equation}\label{eq:orig}
\dot{x} = F_{\mu}(x)
\end{equation}
 has a stationary point at $x=0$ for all $\mu$. In this equation $x \in \fR^n$, $\mu \in \fR^p$ and $F_{\mu}(x)\frac{\partial}{\partial x}$ is a vector field on $\fR^n$ smoothly depending on $x$ and $\mu$. Furthermore suppose that $A: \fR^p \to \gl(n,\fR): \mu \mapsto A(\mu) = DF_{\mu}(0)$ is the linear part of the vector field at $x=0$. We are interested in the case that the phase space is $4$-dimensional and the linear part $A(\mu)$ has a double semi-simple pair of complex conjugate imaginary eigenvalues $(i\beta,-i\beta,i\beta,-i\beta)$, $\beta \neq 0$, at $\mu = 0$. A system with such a linear part is also said to be in 1:1-\emph{resonance}.
With $I$ denoting a two by two identity matrix and $\beta=1$ we have
\begin{equation}\label{eq:l}
A(0) = L = \begin{pmatrix} 0 & I\\-I & 0 \end{pmatrix}.
\end{equation}

Our aim is to give a description of a small neighborhood of $L$ in $\gl(4,\fR)$ and in particular find the stability domain and its boundary. On the latter we expect bifurcations of the stationary point of the original differential equation \eqref{eq:orig}. Thus this study fits in a much larger study of local bifurcations of systems having a stationary point with zero or imaginary eigenvalues in the linear part.

We give a short list of generic bifurcations with increasing number of zero or imaginary eigenvalues in the linear part. We do not specify non-degeneracy conditions, but instead refer to the literature. In dimension one, only one eigenvalue can be zero and we have \emph{transcritical} (TC) and \emph{saddle-node} (SN) bifurcations. However in the presence of symmetry the \emph{pitchfork} (PF) bifurcation occurs. In dimension two there can be two non semi-simple zero eigenvalues denoted by $0^2$, then the system has a \emph{Bogdanov-Takens} (BT) bifurcation. But there can also be a pair of imaginary eigenvalues, then the system undergoes a \emph{Hopf} bifurcation. If the two dimensional system is Hamiltonian and has two zero eigenvalues $0^2$, then there is in general a \emph{Hamiltonian saddle-node} (SSN) bifurcation.

When the dimension gets larger, results get sparser because the codimensions of the bifurcations readily increase. In dimension three, the simplest case is the \emph{Hopf-saddle-node} (HSN) bifurcation when there is one zero eigenvalue and an imaginary pair, denoted $0\beta$. Other cases are systems with linearizations having eigenvalues $0^3$ or $00^2$. The bifurcations for these cases are not yet well studied. The simplest bifurcation for general systems in dimension 4 is the \emph{Hopf-Hopf} (HH) bifurcation where the linearization at the stationary point has two pairs of \emph{non-resonant} imaginary eigenvalues $\beta_1\beta_2$. Other results for $4$-dimensional systems are known for special cases only. For Hamiltonian systems there is a bifurcation at $k$:$l$-resonance of imaginary eigenvalues $\beta_1\beta_2$. A special place is taken by the 1:$\pm 1$-resonances, where the sign in 1:$\pm 1$ designates \emph{symplectic signature} of eigenvalues. At $1$:$-1$-resonance there is a so called \emph{Hamiltonian-Hopf} (SH) bifurcation. But at $1$:$1$-resonance there is another bifurcation that has not been analyzed completely. The $1$:$1$-resonance in reversible systems is very similar to the $1$:$-1$-resonance in Hamiltonian systems, note that there is no signature for imaginary eigenvalues in reversible systems. However, there is in reversible systems a \emph{reversible sign} for real and zero eigenvalues, see \cite{hvn96}. The case $0_+^4$ has been studied, but the case $0_-^4$ has not, as far as we know. We summarize the local bifurcations (with lowest codimension) occurring in systems up to dimension 4 in table \ref{tab:bifs}.

\begin{table}[htbp]
\begin{center}
\begin{tabular}{|c|c|l|c|l|l|}\hline
dim & evc & bif & codim & comments & references\\\hline
1   & $0$ & SN  & 1 & & \cite{gh, kuz}\\
    &     & TC  & 1 & ``$x=0$ stationary for all $\mu$'' & \cite{gh, kuz}\\
    &     & PF  & 2 & & \cite{gh, kuz}\\
    &     & PF  & 1 & $\fZ_2$-symmetric & \cite{gh, kuz}\\\hline
2  & $0^2$   & BT  & 2 & & \cite{drsz}\\
   & $\beta$ & H   & 1 & & \cite{gh, kuz}\\
   & $0^2$   & SSN & 1 & symplectic & \cite{wig}\\\hline
3  & $0\beta$ & HSN & 2 & & \cite{bv, gh, gpd}\\\hline
4  & $\beta_1\beta_2$ & HH & 2 & $\beta_1 : \beta_2$ irrational & \cite{gh, kuz}\\
   & $\beta^2$        & SH & 1 & symplectic 1:$-1$-resonance & \cite{mee}\\
   & $\beta_1\beta_2$ & nn & 1 & symplectic $1$:$2$-resonance  & \cite{dui}\\
   & $\beta_1\beta_2$ & nn & 3 & symplectic $1$:$3$-resonance  & \cite{dui}\\
   & $\beta_1\beta_2$ & nn & 2 & symplectic $k$:$l$-resonance, $k$:$l \neq 1$:$2$, $k$:$l \neq 1$:$3$  & \cite{dui}\\
   & $\beta^2$        & nn & 1 & reversible 1:1-resonance  & \cite{msv}\\
   & $\beta_+\beta_+$ & nn & nk & symplectic 1:1-resonance & \cite{hh}\\
   & $\beta\beta$     & nn & nk & 1:1-resonance &\\
   & $0_+^4$          & nn & 2 & reversible, reversible sign +1 & \cite{ioo}\\\hline
\end{tabular}
\end{center}
\caption{\textit{Low codimension bifurcations of stationary points in systems up to dimension 4. We do not list the non-degeneracy conditions. When such conditions are violated but higher order conditions are met, we generally get a higher codimension bifurcation for the same eigenvalue configuration. The abbreviations have the following meaning dim: dimension of phase space; evc: eigenvalue configuration (see section \ref{sec:stabdombkar}); bif: name of bifurcation; codim: codimension of bifurcation; SN: saddle-node; TC: transcritical; PF: pitchfork; BT: Bogdanov-Takens; SSN: Hamiltonian saddle-node; HSN: Hopf-saddle-node; HH: Hopf-Hopf; SH: Hamiltonian Hopf; nn: no name; nk: not known.}\label{tab:bifs}}
\end{table}

Here we do not even attempt to describe the bifurcation of system \eqref{eq:orig}, whose linear part at $x=0$ and $\mu=0$ is in 1:1-resonance. Instead we focus on the linearized system $\dot{x}=A(\mu)x$ near $\mu=0$. As mentioned before we will describe the neighborhood of $A(0)=L$ in $\gl(4,\fR)$, in particular the stability domain and the singularities on its boundary. Moreover it is of both theoretical and practical interest to know how linear Hamiltonian, reversible and equivariant subsystems appear as subspaces in $\gl(4,\fR)$ and how they intersect the stability domain and its boundary near $L$. This last issue will be treated in more detail in \cite{hkb}.

The \emph{stability domain} in $\gl(4,\fR)$ is defined as follows.
\begin{equation}\label{eq:stabdom}
\stabdom = \{A \in \gl(4,\fR)\;|\; \text{if $\lambda$ is an eigenvalue of $A$ then $\Re(\lambda) < 0$}\}
\end{equation}
Then the boundary of the stability domain $\dstabdom$, is characterized by vanishing real parts of one or more eigenvalues of $A \in \gl(4,\fR)$. By the implicit function theorem, $\dstabdom$ as a hyper surface in $\gl(4,\fR)$, is smooth in points where the corresponding matrix has a simple zero eigenvalue or a simple complex conjugate pair with vanishing real part. At points where multiple eigenvalues have vanishing real parts we may expect singularities. A simple example being two pairs of purely imaginary eigenvalues $\pm i \beta_1$ and $\pm i \beta_2$ with $\beta_1 \neq \beta_2$. Then $\dstabdom$ has generically a transverse self-intersection. We will only consider the stability domain and its boundary in a small neighborhood of $L$ and in figure \ref{fig:eigcfgs} we list the possible \emph{eigenvalue configurations}. For an informal definition of eigenvalue configuration see section \ref{sec:statres}, a precise definition will be given in section \ref{sec:stabdom}. We use the following coding: $\gamma$ for a pair of complex conjugate eigenvalues; $\beta$ for a pair of complex conjugate imaginary eigenvalues; $\alpha$ for a real eigenvalue; $0$ for a zero eigenvalue; $\gamma \gamma$ for semi-simple double eigenvalues and $\gamma^2$ for double eigenvalues with nilpotent part of height 2. A subscript $\pm$ denotes the sign of the real part and an optional index is used to denote different eigenvalues.
\begin{figure}[htbp]
\setlength{\unitlength}{1mm}
\begin{picture}(100,25)(0,0)
\put(10, 15){\includegraphics{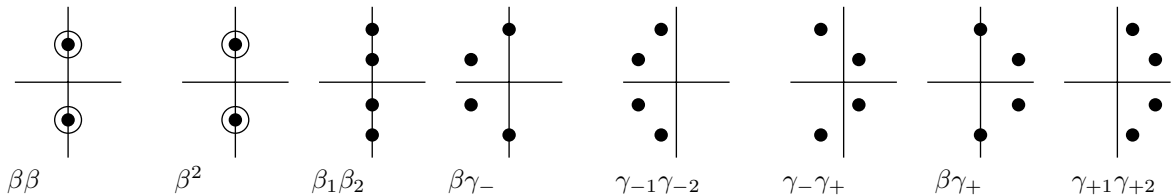}}
\put(2,   1){$\beta\beta$}
\put(24,  1){$\beta^2$}
\put(42,  1){$\beta_1\beta_2$}
\put(60,  1){$\beta\gamma_-$}
\put(82,  1){$\gamma_{-1}\gamma_{-2}$}
\put(104, 1){$\gamma_-\gamma_+$}
\put(124, 1){$\beta\gamma_+$}
\put(142, 1){$\gamma_{+1}\gamma_{+2}$}
\end{picture}
\caption{\textit{Eigenvalue configurations near $L$. $\beta\beta$ is the eigenvalue configuration of $L$. On $\dstabdom$ we have $\beta^2$, $\beta_1\beta_2$ or $\beta\gamma_-$, on $\stabdom$ we only have $\gamma_{-1}\gamma_{-2}$ and elsewhere we have $\gamma_-\gamma_+$, $\beta\gamma_+$ or $\gamma_{+1}\gamma_{+2}$.}\label{fig:eigcfgs}}
\end{figure}
\begin{remark}\label{rem:randcfgs}
There are many more eigenvalue configurations on $\dstabdom$, including zero or real eigenvalues, for example $\alpha_{-1}\alpha_{-2}\beta$, $0\alpha_-\beta$ and $0^4$. In points of $\dstabdom$ where the corresponding matrix has eigenvalue configuration $\alpha_{-1}\alpha_{-2}\beta$, $\dstabdom$ is smooth, for $0\alpha_-\beta$ it has a self-intersection and for $0^4$, $\dstabdom$ will be more singular. In a study of $0^4$ the present study of $L$ will appear as a sub case. However, these eigenvalue configurations do not occur on an arbitrary small neighborhood of $L$. 
\end{remark}

The question of singularities on the boundary of the stability domain has been taken up earlier, see for example a discussion of the \emph{decrement diagram} in \cite{arn1}. There a classification is guided by codimension that is by the number of parameters in a family of matrices $A(\mu)$ where $A(0)$ is the central singularity. Also see \cite{le82,ms99} for an elaboration on this idea with examples. Here we wish to view a study of $L$ in a classification guided by dimension of the phase space. This will lead to high codimensions. Indeed, in studying $L$ we have to consider an eight parameter family. See section \ref{sec:methrcu} how we reduce such families. In the end it turns out that we have to study a three parameter family. In this family we do find most of the singularities listed in \cite{arn1} for generic three parameter families.

\subsection{Motivation and main questions}\label{sec:intromoq}
The main motivation for this study comes from a wide variety of applications where the question of stability of a system near 1:1-resonance turns up in various forms. For example, double semi-simple imaginary eigenvalues are natural in the spectra of rotationally and spherically symmetrical models of solids and fluids \cite{ki09}, ranging from car brakes \cite{ki08}, rotating shafts \cite{nn98}, and computer hard discs \cite{cb92} to rotating elastic Earth \cite{rv09}, and from vortex tubes \cite{hf08} to magneto-hydrodynamics \cite{kgs09}. Double semi-simple eigenvalues are also characteristic of optimal structures and are responsible for high sensitivity of the latter to small imperfections \cite{slo94}. Many \emph{dissipation-induced instabilities} in water wave models can be traced back to the occurrence of double imaginary eigenvalues, see \cite{br97} and references therein, but also \cite{mk91,bkmr94,mo95,km07,ki07,bmr08}. An early observation of friction induced instability can be found in \cite{zie}, which has been related to a singularity on the boundary of the stability domain by \cite{bot}. For more applications and references see \cite{hkb}.

The questions in the applications mentioned above in many instances boil down to questions about the stability domain and its boundary in $\gl(4,\fR)$ near a matrix with double semi-simple imaginary eigenvalues. The main questions we wish to address here are the following.
\begin{enumerate}\itemsep 0pt
\item What are the open domains in parameter space with constant eigenvalue configuration?
\item What are the singularities on the boundaries of these domains?
\item In particular, we address the above two questions for the stability domain.

\commentaar{volgende naar artikel B}

%
\end{enumerate}
We will consider these questions for a general system. But in many applications such a system can be considered as small, dissipative perturbations of a Hamiltonian system in $1:\pm1$-resonance. However, also reversibility and equivariance with respect to a circle group play a prominent role. In all our constructions to study the neighborhood of a matrix with double semi-simple imaginary eigenvalues, we take care to carry them out in such a way that Hamiltonian, reversible and equivariant subsystems can be recognized easily.

\subsection{Organization}\label{sec:introorg}
In section \ref{sec:statres} we give an overview of the results, that is a description of the domains with constant eigenvalue configuration near $L$ in $\gl(4,\fR)$. In particular we present the stability domain and its singularities. The methods we use will be outlined in section \ref{sec:methods}. We give a short overview of the centralizer unfolding of a matrix $A$ in $\gl(n,\fR)$, a family describing a neighborhood of $A$. We also present a method to reduce the number of parameters in this family leading to a reduced centralizer unfolding with the same properties but easier to analyze. This will be applied in section \ref{sec:unfol} to the matrix $L$ introduced in section \ref{sec:introset}. The resulting reduced centralizer unfolding of $L$ is further analyzed in section \ref{sec:stabdom} where we characterize the stability domain and its singularities. Finally in section \ref{sec:evcnbhdl} we give a description of a small neighborhood of $L$ in $\gl(4,\fR)$.

\section{Statement of results}\label{sec:statres}
Our aim is to describe a neighborhood of a matrix with a semi-simple double pair of imaginary eigenvalues in $\gl(4,\fR)$. The dimension of this space is rather large, namely 16, therefore we first apply several reductions to simplify the analysis. Two main ingredients of this reduction are \emph{transversality} and \emph{equivalence classes}. In the following we use some results for smooth group actions. As a general reference see \cite{bre} and \cite{ttd}, for application to dynamical systems see \cite{arn1}.

We use the fact that a linear differential equation transforms like a linear map under a change of coordinates. That is the map $A \in \gl(4,\fR)$ in $\dot{x} = Ax$ transforms like $A \mapsto g^{-1} A g$ under the coordinate transformation $g \in \Gl(4,\fR)$. Thus, in this respect $\Gl(4,\fR)$ acts on $\gl(4,\fR)$ by \emph{similarity transformations}, the action is called the \emph{adjoint action} of $\Gl(4,\fR)$ on $\gl(4,\fR)$. The orbits of the adjoint action are smooth manifolds, but moreover they form equivalence classes.

To reduce a neighborhood of $L$ we consider a complement of the equivalence class of $L$. This runs as follows. In view of the previous, the equivalence class of $L$ has a tangent space. We now take a complement of the tangent space in $L$, ie locally a transverse section of the $\Gl(4,\fR)$-orbit of $L$, see figure \ref{fig:transorb}. In the present context this is also called a \emph{versal unfolding} of $L$. Then for each $A \in \gl(4,\fR)$ in a small neighborhood of $L$, the $\Gl(4,\fR)$-orbit of $A$ transversally intersects this complement. Thus each element $A$ in a neighborhood of $L$ is equivalent to an element in the transverse section of the $\Gl(4,\fR)$-orbit of $L$.

The results below are formulated for a specific versal unfolding of the matrix $L$ as defined in equation \eqref{eq:l}. However we only list properties of $L$ and a neighborhood of $L$ in $\gl(4,\fR)$ that are invariant under smooth changes of coordinates in both the phase space $\fR^4$ and $\gl(4,\fR)$. This means that in the neighborhood of any other matrix in $\gl(4,\fR)$ with eigenvalues $(i\beta,-i\beta,i\beta,-i\beta)$, $\beta \neq 0$, there is a stability domain with singularities on its boundary diffeomorphic to the one described below.
\begin{figure}[htbp]
\setlength{\unitlength}{1mm}
\begin{picture}(100,50)(0,0)
\put(80, 25){\includegraphics{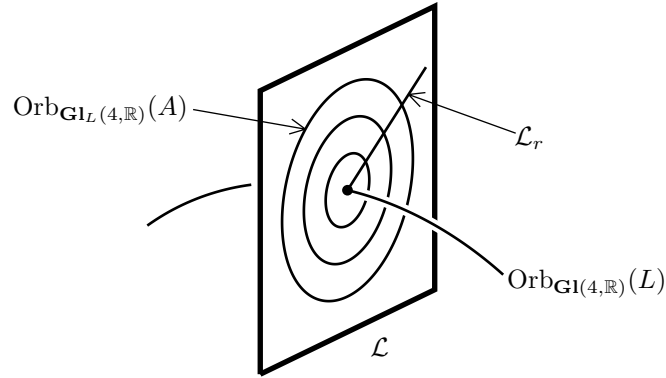}}
\put(101,12){$\Orb_{\Gl(4,\fR)}(L)$}
\put(36, 35){$\Orb_{\Gl_L(4,\fR)}(A)$}
\put(83,  3){$\L$}
\put(102,31){$\L_r$}
\end{picture}
\caption{\textit{Schematic picture of $\gl(4,\fR)$ in a neighborhood of $L$. The centralizer unfolding $\L$ is an orthogonal complement (a particular transverse section) of the tangent space of $\Orb_{\Gl(4,\fR)}(L)$ at $L$. For $A \in \L$, $\Orb_{\Gl_L(4,\fR)}(A)$ is the $\Gl_L(4,\fR)$-orbit of $A$ in $\L$. The reduced centralizer unfolding $\L_r$ is a transverse section of the $\Gl_L(4,\fR)$-orbits in $\L$.}\label{fig:transorb}}
\end{figure}

There are many choices for a versal unfolding. One with particular nice properties is the \emph{centralizer unfolding}, see section \ref{sec:methcu}. In the following result we use a basis in $\gl(4,\fR)$, for details we refer to section \ref{sec:unfolcent}.
\begin{lemma}
Let $\basis{M_1, \ldots, M_8, P_1, \dots, P_8}$ be a basis for $\gl(4,\fR)$, such that $\basis{P_1, \dots, P_8}$ spans the tangent space of the $\Gl(4,\fR)$-orbit of $L$ and $\basis{M_1, \ldots, M_8}$ spans its orthogonal complement. Then the centralizer unfolding of $L$ is given by $\L : \fR^8 \to \gl(4,\fR)$ with
\begin{equation*}
\L(\mu) = L + \sum_{i=1}^8 \mu_i M_i.
\end{equation*}
\end{lemma}
\begin{remark}
In choosing a basis for $\gl(4,\fR)$ we have taken into account that we wish to recognize Hamiltonian, reversible and equivariant subsystems.
\end{remark}
By taking a versal unfolding of $L$ we have reduced the dimension of the original problem from sixteen to eight. The latter is the codimension of the tangent space of the $\Gl(4,\fR)$-orbit of $L$. One of the nice properties of the centralizer unfolding is that there is an easy characterization of the subgroup of $\Gl(4,\fR)$ preserving $\L$, namely $\Gl_L(4,\fR)$ the group of all transformations $g$ commuting with $L$. The $\Gl_L(4,\fR)$-orbits in $\L$ form equivalence classes in $\L$ which allows us to further reduce the dimension of the problem by again taking a transverse section to these orbits. This results in a \emph{reduced centralizer unfolding}, reducing the dimension of problem by three. For details we refer to sections \ref{sec:methrcu} and \ref{sec:unfolredcent}. 
\begin{lemma}
A reduced centralizer unfolding of $L$ is given by $\L_r : \fR^5 \to \gl(4,\fR)$ with
\begin{equation*}
\L_r(\nu) = L + \nu_1 M_1 + \nu_2 M_4 + \nu_3 M_6 + \nu_4 M_8 + \nu_5 M_5.
\end{equation*}
\end{lemma}
See figure \ref{fig:transorb} for a schematic picture of the relation between $\L$ and $\L_r$ in $\gl(4,\fR)$. 
\begin{remark}
There are many possible choices for a reduced centralizer unfolding. This particular choice enables us to easily recognize Hamiltonian, reversible and equivariant subsystems later on.
\end{remark}

Instead of describing the stability domain in the full 16-dimensional neighborhood of $L$ we do this for the reduced unfolding $\L_r$ of $L$, which is still 5-dimensional. The boundary of the stability domain is determined by the property that $\L_r(\nu)$ has at least one pair of purely imaginary eigenvalues. Note that there are no zero eigenvalues in an arbitrary small neighborhood of $L$. This leads to a condition for $\nu$ and we find the following, see section \ref{sec:stabdombkar}.
\begin{lemma}\label{lem:critc}
The boundary of the stability domain of $\L_r$ is contained in the \emph{critical set} $C = \{\nu \in \fR^5 \;|\; F(\nu)=0\}$ where $F(\nu) = (\nu_1^2 - \nu_2^2)(\nu_1^2 + \nu_4^2) + \nu_1^2 \nu_3^2$. 
\end{lemma}
Note that $F$ does not depend on $\nu_5$. This is not surprising when we mention that $M_5 = L$. Thus the problem of describing the stability domain has been reduced to 4 dimensions. One further reduction is possible when we use the fact that $F$ is a homogeneous polynomial. This implies that the boundary of the stability domain transversally intersects the $3$-sphere $||\nu|| = r$. Therefore we restrict the description of the stability domain to this $3$-sphere. This only holds for small $r$, but after an appropriate scaling we may set $r=1$. Then we have the following results, see section \ref{sec:stabdombkar}.
\begin{lemma}\label{lem:crits}
The boundary of the stability domain of the reduced centralizer unfolding $\L_r$ transversally intersects the $3$-sphere in the \emph{critical surface} $S = \{\nu \in \fR^4 \;|\; F(\nu)=0\;\;\text{and}\;\;||\nu||=1\}$. The critical set $C$ is a (straight) cone over the critical surface $S$.
\end{lemma}

With help of a parameterization of the critical surface $S$ we obtain a Whitney stratification of $S$ and thus of the $3$-sphere $||\nu||=1$. On the strata the eigenvalue configuration is constant, so that we are able to identify the stability domain. For a precise definition of \emph{eigenvalue configuration} see section \ref{sec:stabdombkar}, here we give an informal definition. The relevant eigenvalue configurations are shown in figure \ref{fig:eigcfgs}.
\begin{definition}
An eigenvalue configuration is informally defined as an equivalence class of sets of eigenvalues such that in any two sets from the same class there are corresponding eigenvalues with negative, zero or positive real parts. 
\end{definition}
We consider the $3$-sphere as two $3$-discs, labelled $\Dp$ and $\Dm$, glued smoothly along their boundaries, which are $2$-spheres. Then the critical surface also has two parts $\Sp$ and $\Sm$, one in each $3$-disc. We do this in such a way that the critical surface transversely intersects the boundary of the $3$-disc. The parameterization of $\Spm$ is a map $\phipm: [-1,1] \times [0,2\pi] \to \Spm : (s,t) \mapsto \phipm(s,t)$. For details we refer to section \ref{sec:stabdombkar}.

\begin{figure}[htbp]
\setlength{\unitlength}{1mm}
\begin{picture}(100,50)(0,0)
\put(15,30){\includegraphics{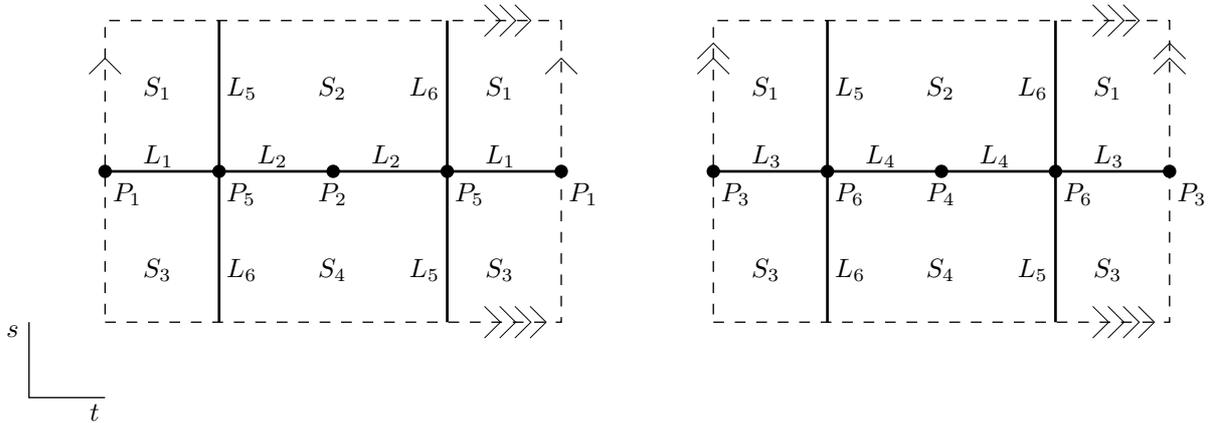}}
\put(20,31){$L_1$}
\put(35,31){$L_2$}
\put(65,31){$L_1$}
\put(50,31){$L_2$}
\put(31,40){$L_5$}
\put(55,40){$L_6$}
\put(55,16){$L_5$}
\put(31,16){$L_6$}
\put(16,26){$P_1$}
\put(76,26){$P_1$}
\put(43,26){$P_2$}
\put(31,26){$P_5$}
\put(61,26){$P_5$}
\put(20,40){$S_1$}
\put(65,40){$S_1$}
\put(43,40){$S_2$}
\put(20,16){$S_3$}
\put(65,16){$S_3$}
\put(43,16){$S_4$}
\put(100,31){$L_3$}
\put(115,31){$L_4$}
\put(145,31){$L_3$}
\put(130,31){$L_4$}
\put(111,40){$L_5$}
\put(135,40){$L_6$}
\put(135,16){$L_5$}
\put(111,16){$L_6$}
\put(96,26){$P_3$}
\put(156,26){$P_3$}
\put(123,26){$P_4$}
\put(111,26){$P_6$}
\put(141,26){$P_6$}
\put(100,40){$S_1$}
\put(145,40){$S_1$}
\put(123,40){$S_2$}
\put(100,16){$S_3$}
\put(145,16){$S_3$}
\put(123,16){$S_4$}
\put(2,  8){$s$}
\put(13,-3){$t$}
\end{picture}
\caption{\textit{Domains of the map $\phipm$. On the left the domain for $\Sp$, on the right the domain for $\Sm$. The images of the domains are glued smoothly along the dashed lines according to the arrows. Lines with three and four arrows correspond to the great circles drawn in figure \ref{fig:conoid}. The 2-dimensional open parts are denoted $S_1,\ldots,S_4$, the 1-dimensional open parts $L_1,\ldots,L_6$. There are self-intersections along $L_1,\ldots,L_6$. Singular points are $P_5$ and $P_6$ where lines of self-intersection intersect. The points $P_1,\ldots,P_4$ correspond to Whitney umbrellas. On the thick lines $\phipm$ is two-to-one, therefore lines and points with the same name are identified in the image.}\label{fig:domain}}
\end{figure}

The domains of the maps $\phipm$ are shown in figure \ref{fig:domain}. In the domains we have points $P_1,\ldots,P_6$ on lines $L_1,\ldots,L_6$. On these lines the maps $\phipm$ are two-to-one. To avoid heavy notation we denote the images under $\phipm$ of the various parts in the domains by the same names. The boundaries of the domains (dashed lines in figure \ref{fig:domain}) are glued by the maps $\phipm$, according to the arrows in figure \ref{fig:domain}. From this we see that for example the open $2$-dimensional part $S_2$ is bounded by $P_5 \cup L_5 \cup P_6 \cup L_6$ whose image is a (great) circle in the $3$-sphere. Thus the image of $S_2$ is a topological $2$-disc. The same conclusion holds for $S_1$, $S_3$ and $S_4$. Furthermore $S_1$ and $S_4$ are glued along the circle $P_5 \cup L_5 \cup P_6 \cup L_6$ forming a topological $2$-sphere which encloses a topological $3$-disc. The latter we call $V_1$. The $3$-discs $V_2$, $V_3$ and $V_4$ are constructed in the same way, also see figure \ref{fig:incid}.

This leads to a Whitney stratification of the $3$-sphere. If we consider points as 0-discs, all strata are (topological) discs. The strata are $V_1,\ldots,V_4$, $S_1,\ldots,S_4$, $L_1,\ldots,L_6$, $P_1,\ldots,P_6$. In figure \ref{fig:conoid} the strata on $\Dp$ are shown. A similar figure for $\Dm$ can be given, see figure \ref{fig:conoids}. The three dimensional strata are not explicitly drawn in the figures. We have the following result on the stratification of the $3$-sphere and the critical surface $S$.
\begin{theorem}
The collection $\{V_1,\ldots,V_4, S_1,\ldots,S_4, L_1,\ldots,L_6, P_1,\ldots,P_6\}$ is a Whitney stratification of the $3$-sphere. The incidence diagram is given in figure \ref{fig:incid}. The sub collection $\{S_1,\ldots,S_4, L_1,\ldots,L_6, P_1,\ldots,P_6\}$ is a Whitney stratification of the critical surface $S$. The eigenvalue configuration on the strata is given in table \ref{tab:eigconfs}. In particular we find that $V_3$ is the stability domain on the $3$-sphere.
\end{theorem}

\begin{figure}[htbp]
\setlength{\unitlength}{1mm}
\begin{picture}(100,70)(10,-5)
\put(30,0){\includegraphics{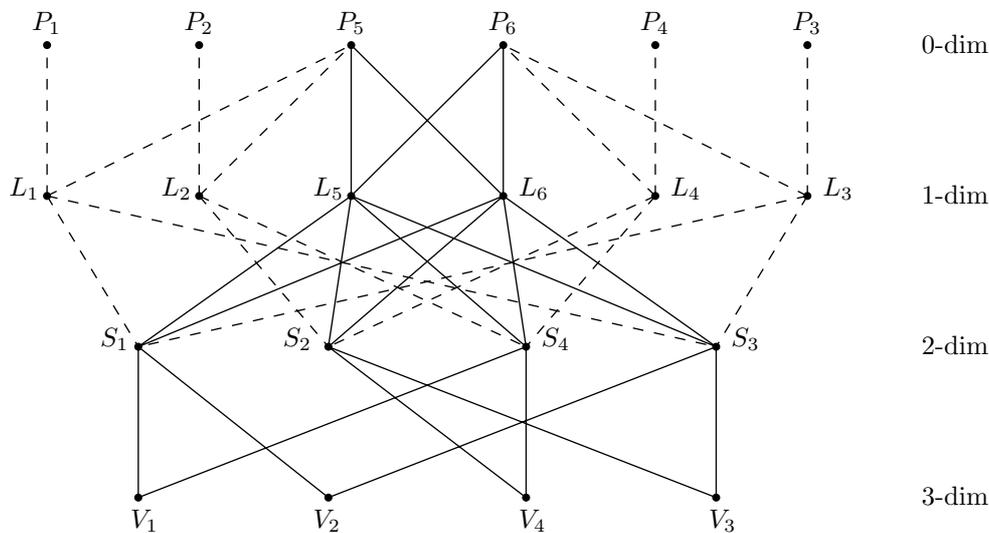}}
\put(28, 62){$P_1$}
\put(48, 62){$P_2$}
\put(68, 62){$P_5$}
\put(88, 62){$P_6$}
\put(108,62){$P_4$}
\put(128,62){$P_3$}
\put(25, 40){$L_1$}
\put(45, 40){$L_2$}
\put(65, 40){$L_5$}
\put(92, 40){$L_6$}
\put(112,40){$L_4$}
\put(132,40){$L_3$}
\put(37, 20){$S_1$}
\put(61, 20){$S_2$}
\put(95, 20){$S_4$}
\put(120,20){$S_3$}
\put(41, -4){$V_1$}
\put(65, -4){$V_2$}
\put(92, -4){$V_4$}
\put(117,-4){$V_3$}
\put(145,59){0-dim}
\put(145,39){1-dim}
\put(145,19){2-dim}
\put(145,-1){3-dim}
\end{picture}
\caption{\textit{Incidence diagram of the stratification of the 3-sphere determined by the critical surface $S$. Only incidences of $n$- and $n$+1-dimensional strata are shown. The solid lines form the simplified incidence diagram used to determine the global structure of $S$. All strata are discs.} \label{fig:incid}}
\end{figure}

\begin{remark}
The dashed lines in the incidence diagram connect points and lines to the surfaces $S_i$ that are not important for the global structure of the critical surface. The reason is that for example $P_1 \cup L_1 \cup P_5$ is a line segment on $\cl{S_1}$ and $\cl{S_3}$ and $P_1$ is only connected to $L_1$. This means that shrinking $L_1$, so that $P_1$ goes to $P_5$, does not change global connections.
\end{remark}

\begin{remark}
Since a bifurcation set is in general a semi-algebraic set it admits a Whitney stratification, see for example \cite{kal}. This is of great help in organizing the parameter space. On each stratum, systems are similar in some sense. To be more precise in our example on each stratum, systems have the same eigenvalue configuration. For an example where the organization of the parameter space is far more intricate and how the Whitney stratification provides structure, see \cite{bhvv}.
\end{remark}

\begin{table}[htbp]
\begin{center}
\begin{tabular}{l|l}
strata           & evc\\\hline
$V_1$            & $\gamma_{+1} \gamma_{+2}$\\
$V_3$            & $\gamma_{-1} \gamma_{-2}$\\
$V_2, V_4$       & $\gamma_+ \gamma_-$\\\hline
$S_1, S_4$       & $\beta \gamma_+$\\
$S_2, S_3$       & $\beta \gamma_-$\\\hline
$L_1,\ldots,L_6$ & $\beta_1 \beta_2$\\\hline
$P_1,\ldots,P_4$ & $\beta^2$\\
$P_5, P_6$       & $\beta_1 \beta_2$\\
\end{tabular}
\end{center}
\caption{\textit{Eigenvalue configurations on the strata $V_1,\ldots,V_4$, $S_1,\ldots,S_4$, $L_1,\ldots,L_6$, $P_1,\ldots,P_6$ near $L$.}\label{tab:eigconfs}}
\end{table}

So far for the global structure of the critical surface. On the critical surface $S$ we find several singular points and curves. In figure \ref{fig:sings} we show the singularities at the points $P_1,\dots,P_4$ and $P_5$ and $P_6$.
\begin{figure}[htbp]
\setlength{\unitlength}{1mm}
\begin{picture}(100,45)(0,0)
\put(10,-10){\includegraphics[scale=0.4]{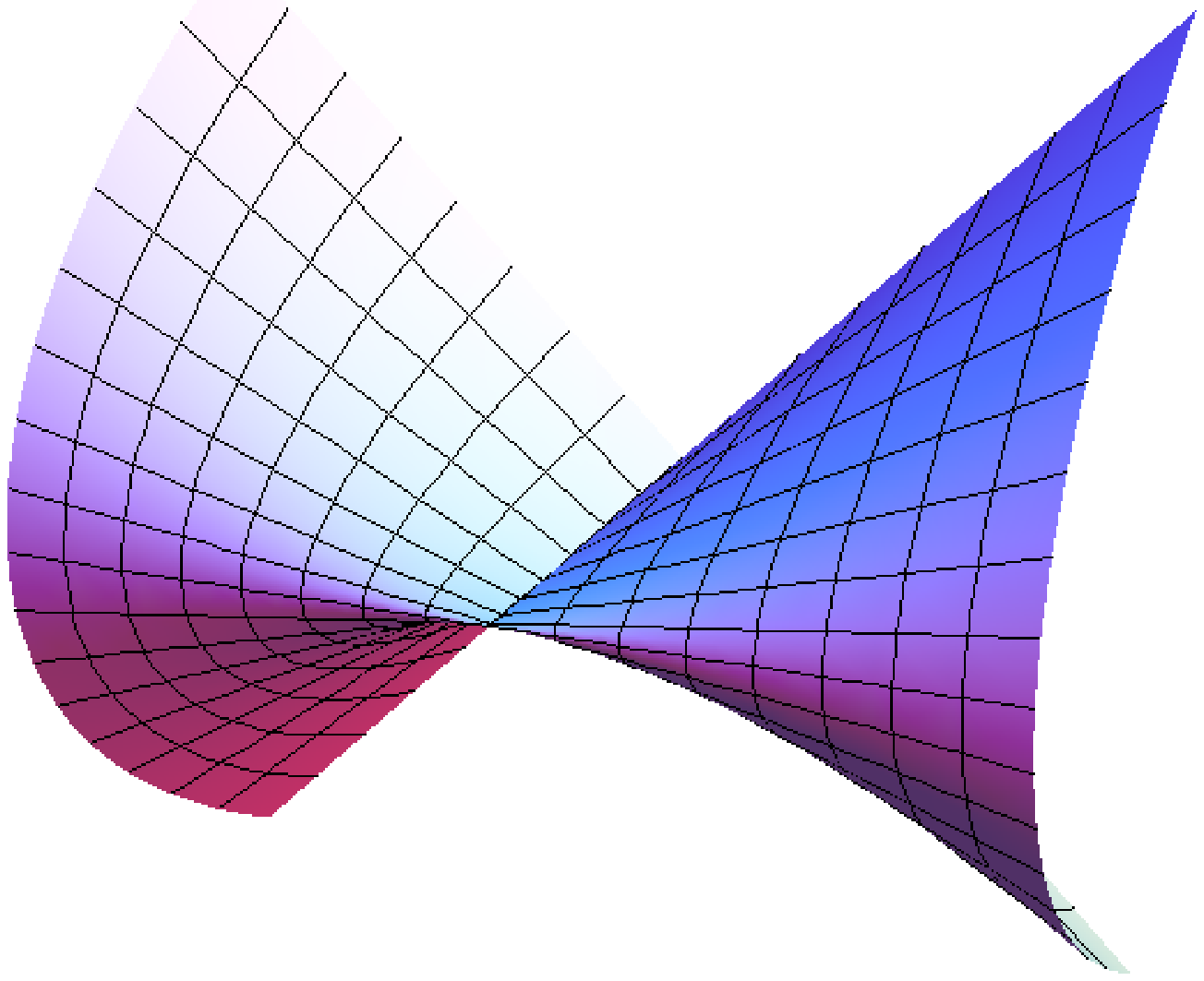}}
\put(90,  0){\includegraphics[scale=0.4]{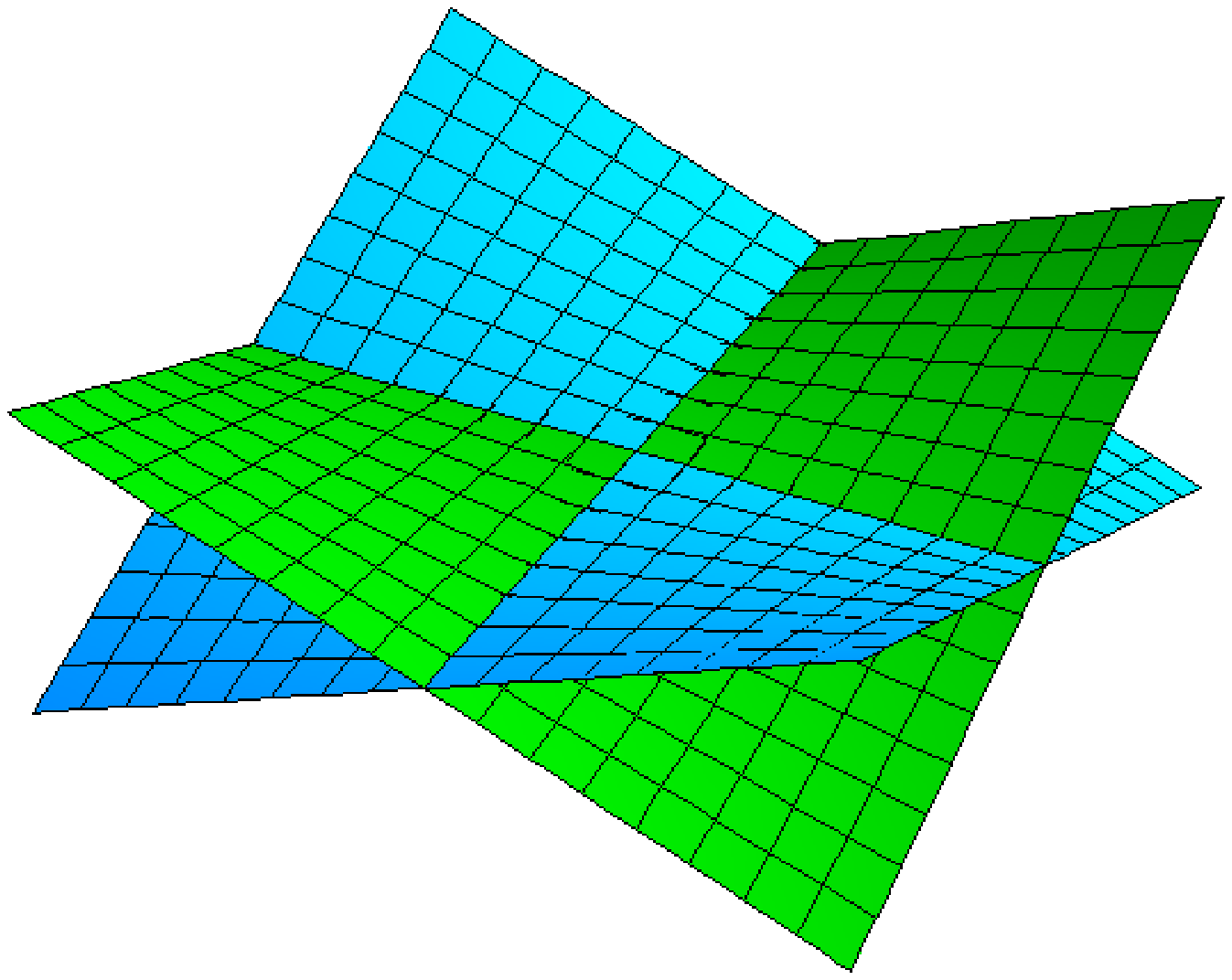}}
\end{picture}
\caption{\textit{Singularities of the critical surface $S$. Left a Whitney umbrella with standard form $x^2=yz^2$. Right an intersection of self-intersections with standard form $z^2=x^2y^2$ (or equivalently $z^2=xyz$).}\label{fig:sings}}
\end{figure}

\begin{theorem}
On the lines $L_1,\ldots,L_6$ the critical surface $S$ has a transverse self-intersection. In the points $P_1,\ldots,P_6$ the self-intersection is non-transverse. In each of the points $P_1,\ldots,P_4$ the critical surface $S$ has a singularity called \emph{Whitney umbrella}. But in the points $P_5$ and $P_6$ curves of self-intersection transversely intersect and the critical surface $S$ has a self-intersection with coinciding tangent planes of saddle type.
\end{theorem}

The stability domain of the reduced centralizer unfolding is $V_3$. Let us summarize the singularities on the boundary of $V_3$ in connection with the classification in \cite{arn1}. On the lines $L_1, \ldots, L_6$ we have transverse self-intersections with standard form $xy=0$. In each of the points $P_1,\ldots,P_4$ we have a Whitney umbrella with standard form $x^2=yz^2$ and in the points $P_5$ and $P_6$ we have an intersection of self-intersections with standard form $z^2=x^2y^2$.

We conclude with two pictures showing both global and local aspects of the critical surface $S$. Figure \ref{fig:conoid} shows the restriction $\Sp$ of the critical surface to one of the $3$-discs, namely to $\Dp$. In fact it is a projection onto one of the coordinate hyper planes ($\nu_3=0$). In figure \ref{fig:conoids} we show a diagram of the global connection of the conoids. Since the critical surface $S$ is on the $3$-sphere we can only give a schematic picture.

\begin{remark}
It is worth noting that as the image of $\phip$, $\Sp$ (and also $\Sm$ as the image of $\phim$) is a ruled surface. The rules are parallel to the $\nu_1,\nu_4$-plane and all of them pass through a curve, the $\nu_2$-axis. Such a surface is called a \emph{Catalan surface}. In our case the curve is a straight line, then the surface is called a \emph{right conoid}. In particular $\Sp$ is known as Pl\"ucker's conoid for $n=1$, where $n$ is the index of the family $(s,t) \mapsto (s\,\cos\,t, \cos\,nt, s\,\sin\,t)$ of parameterizations of Pl\"ucker conoids, see \cite{bg, gas}.
\end{remark}

\begin{figure}[htbp]
\setlength{\unitlength}{1mm}
\begin{picture}(100,77)(0,0)
\put(10,-10){\includegraphics[scale=0.6]{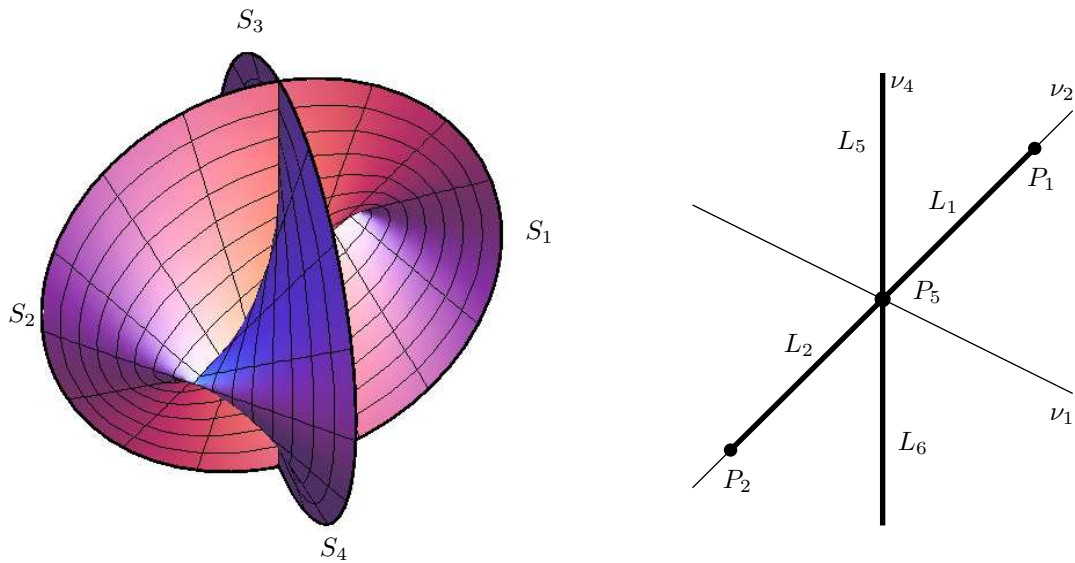}}
\put(125,35){\includegraphics{skelet1.ps}}
\put(147,19){$\nu_1$}
\put(147,62){$\nu_2$}
\put(126,63){$\nu_4$}
\put(131,47){$L_1$}
\put(112,28){$L_2$}
\put(119,55){$L_5$}
\put(127,15){$L_6$}
\put(144,50){$P_1$}
\put(104,10){$P_2$}
\put(129,35){$P_5$}
\put( 78,43){$S_1$}
\put( 10,32){$S_2$}
\put( 40,71){$S_3$}
\put( 51, 1){$S_4$}
\end{picture}
\caption{\textit{On the left the surface $\Sp$, a right conoid, consisting of the open $2$-dimensional strata $S_1,\ldots,S_4$ is shown. On the right we separately draw the lines of self-intersection $L_1$, $L_2$, $L_5$ and $L_6$ and the singular points $P_1$, $P_2$ and $P_5$. At points $P_1$ and $P_2$ we find Whitney umbrellas and at $P_5$ we find an intersection of lines of self-intersection.}\label{fig:conoid}}
\end{figure}

\begin{figure}[htbp]
\setlength{\unitlength}{1mm}
\begin{picture}(100,70)(0,0)
\put(0, -25){\includegraphics[scale=0.7]{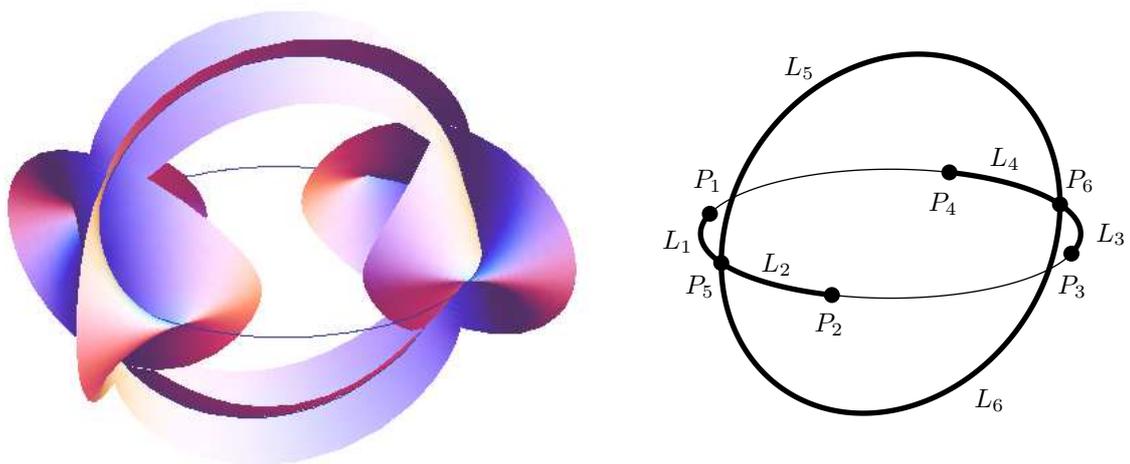}}
\put(125,35){\includegraphics{skelet2.ps}}
\put(99, 41){$P_1$}
\put(115,22){$P_2$}
\put(147,27){$P_3$}
\put(130,38){$P_4$}
\put(98, 27){$P_5$}
\put(148,41){$P_6$}
\put(95, 33){$L_1$}
\put(108,30){$L_2$}
\put(152,34){$L_3$}
\put(138,44){$L_4$}
\put(111,56){$L_5$}
\put(136,12){$L_6$}
\end{picture}
\caption{\textit{Diagram of the global connection of the right conoids $\Sp$ and $\Sm$. On the left the two right conoids are shown together with their connection along a circle of self-intersections. On the right we show the singular curves and points. The circle $P_5 \cup L_5 \cup P_6 \cup L_6$ is a great circle on the $3$-sphere and also the circle containing the segments $P_1 \cup L_1 \cup P_5 \cup L_2 \cup P_2$ and $P_3 \cup L_3 \cup P_6 \cup L_4 \cup P_4$ is a great circle on the $3$-sphere. In the points $P_1$, $P_2$, $P_3$ and $P_4$ we have Whitney umbrellas and in the points $P_5$ and $P_6$ we have intersections of self-intersections..}\label{fig:conoids}}
\end{figure}

\section{Methods}\label{sec:methods}
%

\subsection{The centralizer unfolding}\label{sec:methcu}
Given a map $L \in \gl(V)$ where $V$ is a real vector space. We are interested in all maps in an open neighborhood of $L$ in $\gl(V)$. Some of those will be elements of the $\Gl(V)$-orbit
\begin{equation*}
\Orb_{\Gl(V)}(L) = \{g^{-1}Lg \;|\; g \in \Gl(V)\}
\end{equation*}
of $L$ and are thus equivalent to/with $L$. A systematic way to explore the neighborhood of $L$ in $\gl(V)$ is given by Arnol'd \cite{arn1}. Here we follow the description given in \cite{hlr1}. At several places we will tacitly use the fact the $\Gl(V)$-orbit of $L \in \gl(V)$ is a smooth subset of $\gl(V)$, see \cite{bre,ttd}.
\begin{definition}
A smooth map $\L: \fR^p \rightarrow \gl(V): \mu \mapsto \L(\mu)$ with $\L(0)=L$ is called an \emph{unfolding} or a \emph{deformation} of $L$. If $\L$ is transverse to the $\Gl(V)$-orbit through $L$ at $L$, then it is said to be \emph{versal}.
\end{definition}
We are especially interested in versal unfoldings having a minimum number of parameters. The main source of these concepts and ideas is Arnol'd \cite{arn1}, which we review in the following.
\begin{definition}
Two unfoldings $A(\mu)$ and $B(\mu)$ of $L$ are called \emph{equivalent} if they are similar as families of linear maps. This means that there is a smooth family of transformations $g(\mu) \in \Gl(V)$ such that $g(\mu)A(\mu)g(\mu)^{-1} = B(\mu)$ for all $\mu \in \fR^p$. An unfolding $\L$ of $L$ is called \emph{miniversal} if (a) for every other unfolding $A: \fR^q \rightarrow \gl(V)$ of $L$ there exists a smooth map $\chi: \fR^q \rightarrow \fR^p$ such that $A$ is equivalent to $\L \circ \chi$, and (b) $\L$ has the minimal number of parameters possible for unfoldings with this property.
\end{definition}
The number of parameters for a miniversal unfolding is equal to the codimension of the $\Gl(V)$-orbit through $L$ and so is called the \emph{codimension} of $L$. Arnol'd \cite{arn1} showed that miniversal unfoldings can be obtained by taking orthogonal complements to tangent spaces of $\Gl(V)$-orbits. Such unfoldings are called \emph{centralizer unfoldings}. To define these we first need an inner product on $\gl(V)$. The proof of the following lemma is rather straightforward.
\begin{lemma}\label{lem:iprod}
The bilinear form $\inprod{A}{B} = \trace(A^*B)$, for $A, B \in \gl(V)$, is an inner product.
\end{lemma}
With this inner product we have the next result.
\begin{proposition}\label{pro:minfo}
The subset $\{L + M^*\;|\;[L,M] = 0\}$ of $\gl(V)$ is a miniversal unfolding of $L$.
\end{proposition}
\begin{proof}
A short computation shows that $M$ is an element of the orthogonal complement of the tangent space of $\Orb_{\Gl(V)}(L)$ if $[L,M^*] = 0$, where $[A,B]=AB-BA$.
\end{proof}
Let $\m = \{M \in \gl(V)\;|\;[L,M]=0\}$ then it is easily seen that $\m$ is a Lie-algebra. We now define the centralizer unfolding as follows.
\begin{definition}\label{def:centunfo}
The \emph{centralizer unfolding} of $L$ is $\{L + M^* \;|\; M \in \m \}$.
\end{definition}
By applying the adjoint action of $\Gl(V)$ on $\gl(V)$ to this unfolding we obtain an unfolding at any other point on the $\Gl(V)$-orbit through $L$. Transversality and miniversality are preserved by this transformation, but orthogonality will usually be lost.

The centralizer unfolding has nice properties. Let $\Gl_L(V) = \{g \in \Gl(V) \;|\; gL=Lg \}$, then the adjoint action of $\Gl_L(V)$ preserves $\m$. Consequently there is a Lie-subgroup $\M \subset \Gl_L(V)$ such that the adjoint action of $\M$ preserves the centralizer unfolding. This property allows us to find equivalence classes in $\L$.

\begin{remark}\label{rem:genlie}
The construction we just described can also be applied to any Lie sub-algebra $\g \subset \gl(V)$ with the associated Lie subgroup $\G \subset \Gl(V)$ as transformation group. We can further generalize this to subsets $\g$ of $\gl(V)$ that are not necessarily Lie algebras but still preserved by the adjoint action of some Lie subgroup $\G$ of $\Gl(V)$, see \cite{hlr1}.
\end{remark}

\begin{remark}\label{rem:lasymm}
In the example we are considering in this article we have $L^*=-L$. In a slightly more general situation when $L^* = \pm L$ we have $\M = \Gl_L(V)$.
\end{remark}

\subsection{The reduced centralizer unfolding}\label{sec:methrcu}
The centralizer unfolding of $L$ is the orthogonal complement of the tangent space of $\Orb_{\Gl(V)}(L)$. In fact, as we already noted in section \ref{sec:methcu}, it is a Lie sub-algebra $\m$ of $\gl(V)$. Then there is a Lie subgroup, which we denote by $\M$, of $\Gl(V)$ which acts on this complement. The Lie algebra of $\M$ is $\m$. This means that $\m$ is invariant under the $\Ad_{\M}$-action $\phi : \M \times \m \to \m : (g, A) \mapsto g^{-1} A g$. Thus it is possible to find equivalence classes in the unfolding of $L$. The orbit space of this group will have a lower dimension than the original unfolding, thus simplifying the analysis.

However, there are two reasons not to take the quotient with the full group. The first is that it might not be compact so that the resulting quotient space is no longer Hausdorff. Therefore we should at least restrict to a compact subgroup (preferably the largest). The second reason is that even when restricting to a compact subgroup, the quotient space may have singularities if the action $\phi$ is non-free. Since our goal is to find the singularities of the boundary of the stability domain in parameter space, we want to avoid introducing singularities by taking a quotient. A general theorem for compact group actions tells us that the orbit space of a smooth, proper and free action $\phi$ is again a smooth manifold, see \cite{bre,ttd}. In our case we call the quotient space $\m / \sim\phi$ the \emph{reduced unfolding} of $L$. The details very much depend on $L$. Therefore we refer to section \ref{sec:unfolredcent} for the group $\M$, its maximal compact subgroup and the quotient space $\m / \sim\phi$ for our particular choice of $L$.

\subsection{Characterization of the stability domain}\label{sec:methstabdomkar}
The stability domain $\stabdom$ near $L$ is an open set in $\gl(4,\fR)$, which is most easily characterized by its boundary $\dstabdom$. However, we do not directly characterize $\dstabdom$, but instead determine the set of $A \in \gl(4,\fR)$ near $L$ where $A$ has at least one pair of imaginary eigenvalues. We describe our method for $n$ dimensions, but slightly specialized for our situation.

First we define a map $\psi $ from sets of roots to real polynomials exactly having these roots. The map $\psi$ parameterizes surfaces in the space of polynomials. Then we characterize the surface where polynomials have at least one pair of imaginary roots as the zero set of a function $f$. Next we define a map $\phi$, mapping elements of the unfolding of $L$ to eigenvalue polynomials. Requiring that the eigenvalue polynomials lie on the surface parameterized by $\psi$, we obtain a surface characterized as the zero set of $f \circ \phi$. This surface contains the boundary of the stability domain, see figure \ref{fig:phipsi}.
\begin{figure}[htbp]
\setlength{\unitlength}{1mm}
\begin{picture}(100,22)(0,0)
\put(0,0){\includegraphics{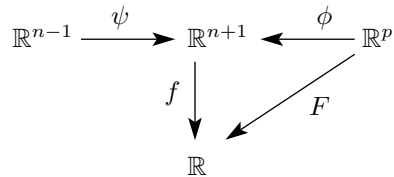}}
\put(52, 17){$\fR^{n-1}$}
\put(75, 17){$\fR^{n+1}$}
\put(98, 17){$\fR^p$}
\put(75,  0){$\fR$}
\put(65, 20){$\psi$}
\put(92, 20){$\phi$}
\put(72, 10){$f$}
\put(91,  8){$F$}
\end{picture}
\caption{\textit{$\psi$ maps sets of roots with at least one imaginary pair to $n$-th degree polynomials, $\phi$ maps elements of an unfolding to $n$-th degree eigenvalue polynomials. The zero set of $f$ consists of polynomials with at least one imaginary root pair.}\label{fig:phipsi}}
\end{figure}

We identify the sets of roots containing at least one imaginary pair with $\fR^{n-1}$. The space of $n$th-degree real polynomials can be identified with $\fR^{n+1}$, namely $a=(a_n, \ldots, a_0)\in \fR^{n+1}$ is in one to one correspondence with $p(\lambda) = a_n \lambda^n + a_{n-1} \lambda^{n-1} + \cdots + a_1 \lambda + a_0$. Furthermore we identify the $p$-parameter homogeneous (reduced centralizer) unfolding $\H(\nu)$ of $L$ with $\fR^p$, see section \ref{sec:unfolcent} for details.

Now we define the map $\psi : \fR^{n-1} \to \fR^{n+1}$ by $\psi(\sigma) = a$. This map is quasi-homogeneous because the coefficients $a$ are homogeneous polynomials of the roots with different degrees. In particular if $\psi(\sigma)=a$, we have for all $t \in \fR_{>0}$
\begin{equation*}
\psi(t \sigma) = (t^n a_0, t^{n-1}a_1,\ldots,ta_{n-1},a_n).
\end{equation*}
We will usually set $a_n=1$, then the map $\psi$ parameterizes a hyper-surface in $\fR^n$. The latter represents polynomials with at least one imaginary pair as roots. Using the Buchberger algorithm \cite{cls} we can eliminate $\sigma$ from $\psi(\sigma)=a$ and obtain an implicit equation $f(a)=0$ where $f:\fR^{n+1} \to \fR$ is a homogeneous polynomial. If $\psi(\sigma)$ is an element of the hyper-surface, then $\psi(t \sigma)$ is also an element of the hyper-surface for all $t \in \fR_{>0}$. Therefore $f$ is also quasi-homogeneous and we have for all $t \in \fR_{>0}$
\begin{equation*}
f(t^n a_0, t^{n-1}a_1,\ldots,ta_{n-1},a_n) = t^k f(a),
\end{equation*}
for some integer $k$.\commentaar{ik zou van te voren willen weten wat k is}

Let $\phi: \fR^p \to \fR^{n+1}$ be the map that maps unfolding parameters $\nu$ to the coefficients $a \in \fR^{n+1}$ of the eigenvalue polynomial of $\H(\nu)$. Then $\phi$ is quasi-homogeneous and in particular if $\phi(\nu)=a$, we have for all $t \in \fR_{>0}$
\begin{equation*}
\phi(t \nu) = (t^n a_0, t^{n-1}a_1,\ldots,ta_{n-1},a_n).
\end{equation*}
Note that since $\phi(\nu)$ are the coefficients of an eigenvalue polynomial, we always have $a_n=1$.

In order that $\H(\nu)$ has at least one pair of imaginary eigenvalues, the coefficients $a$ of its eigenvalue polynomial must lie on the hyper-surface defined by $f(a)=0$. Thus we find that the parameters $\nu$ must satisfy $f(\phi(\nu))=0$. Therefore we now define $F:\fR^p \to \fR$ as
\begin{equation*}
F = f \circ \phi. 
\end{equation*}
The zero set of $F$ determines a hyper-surface in the homogeneous (reduced centralizer) unfolding of $L$, which contains the stability boundary. From the homogeneity properties of $\psi$, $f$ and $\phi$ it follows that $F$ is homogeneous of degree $k$, with the same $k$ as above, for all $t \in \fR_{>0}$
\begin{equation*}
F(t \nu) = f(\phi(t \nu)) = t^k F(\nu).
\end{equation*}
\begin{remark}
On the boundary of the stability domain, $a$ defined by $\phi(\nu) = a$ must satisfy $f(a) = 0$ and $a_n = 1$. For our purposes it is not enough to look at the singularities of the hyper surface $f(a)= 0$ because in the space of coefficients of eigenvalue polynomials we only have algebraic information of the eigenvalues. In the space of unfolding parameters we also have geometric information about eigenvalues. Thus the hyper surface defined by $F(\nu) = 0$ has possibly more singularities than the hyper surface $f(a) = 0$. For example, in the latter we will not see the difference between semi-simple and non-semi-simple eigenvalues.
\end{remark}
\begin{remark}
In view of the previous remark it would be interesting to know the fibers of the map $\phi$.
\end{remark}

\section{Unfolding of $L$}\label{sec:unfol}
%

\subsection{The centralizer unfolding of $L$}\label{sec:unfolcent}
Here we apply the results of section \ref{sec:methcu} to the map $L \in \gl(4,\fR)$ from section \ref{sec:intro}, which has a double pair of eigenvalues $\pm i$. From the zero commutator criterion in proposition \ref{pro:minfo} the centralizer unfolding of $L$ is readily seen to be
\begin{equation*}
\L(\mu) = \begin{pmatrix} 0 & I\\-I & 0 \end{pmatrix} +
\begin{pmatrix} A(\mu) & B(\mu)\\-B(\mu) & A(\mu) \end{pmatrix}.
\end{equation*}
$A(\mu)$ and $B(\mu)$ are any $2 \times 2$ matrices. Thus the centralizer unfolding has eight parameters and consequently every other miniversal unfolding of $L$ must contain eight parameters. To further characterize the unfolding we now choose a basis for $\m = \{M \in \gl(4,\fR)\;|\;[L,M]=0\}$ and $\gl(4,\fR)$, regarded as a linear space. First we introduce
\begin{equation}\label{eq:twomatrices}
I = \begin{pmatrix} 1 & 0 \\ 0 & 1 \end{pmatrix}, \quad
R = \begin{pmatrix} 1 & 0 \\ 0 &-1 \end{pmatrix}, \quad
T = \begin{pmatrix} 0 & 1 \\ 1 & 0 \end{pmatrix}, \quad
J = \begin{pmatrix} 0 & 1 \\-1 & 0 \end{pmatrix},
\end{equation}
as a basis for the space of $2 \times 2$ matrices, with help of these we define the following 16 matrices

\begin{equation}\label{eq:emmen}
\begin{array}{llll}
M_1 = \begin{pmatrix} I & 0 \\ 0 & I \end{pmatrix},&
M_2 = \begin{pmatrix} R & 0 \\ 0 & R \end{pmatrix},&
M_3 = \begin{pmatrix} T & 0 \\ 0 & T \end{pmatrix},&
M_4 = \begin{pmatrix} 0 & J \\-J & 0 \end{pmatrix},\\[3ex]
M_5 = \begin{pmatrix} 0 & I \\-I & 0 \end{pmatrix},&
M_6 = \begin{pmatrix} 0 & R \\-R & 0 \end{pmatrix},&
M_7 = \begin{pmatrix} 0 & T \\-T & 0 \end{pmatrix},&
M_8 = \begin{pmatrix} J & 0 \\ 0 & J \end{pmatrix}
\end{array}
\end{equation}
and
\begin{equation}\label{eq:matp}
\begin{array}{llll}
P_1 = \begin{pmatrix} I & 0 \\ 0 &-I \end{pmatrix},&
P_2 = \begin{pmatrix} R & 0 \\ 0 &-R \end{pmatrix},&
P_3 = \begin{pmatrix} T & 0 \\ 0 &-T \end{pmatrix},&
P_4 = \begin{pmatrix} 0 & J \\ J & 0 \end{pmatrix},\\[3ex]
P_5 = \begin{pmatrix} 0 & I \\ I & 0 \end{pmatrix},&
P_6 = \begin{pmatrix} 0 & R \\ R & 0 \end{pmatrix},&
P_7 = \begin{pmatrix} 0 & T \\ T & 0 \end{pmatrix},&
P_8 = \begin{pmatrix} J & 0 \\ 0 &-J \end{pmatrix}.
\end{array}
\end{equation}
Note that $M_1$ is equal to the identity and $M_5 = L$. The following properties are easily checked.
\begin{lemma}\label{lem:basisgl4}
The set $\basis{M_1,\ldots,M_8,P_1,\ldots,P_8}$ is a basis for $\gl(4,\fR)$.
\end{lemma}
\begin{lemma}\label{lem:basism}
The set $\basis{M_1,\ldots,M_8}$ is a basis for $\m$.
\end{lemma}
With these two lemmas we immediately get the main result of this section about the centralizer unfolding of $L$.
\begin{proposition}\label{prop:unfol}
$\L(\mu) = L + \sum_{i=1}^8 \mu_i M_i$, with parameters $\mu_1,\ldots,\mu_8 \in \fR$ and matrices $M_i$ as in equation \eqref{eq:emmen}, is a centralizer unfolding of $L$. The codimension of the unfolding is 8.
\end{proposition}
The centralizer unfolding of $L$ can be regarded as a linear variety in $\gl(4,\fR)$, the sum of the fixed vector $L$ and the linear space $\m$. It will turn out to be useful to have the following notion.
\begin{definition}\label{def:homunf}
We will call $\H(\mu) = \sum_{i=1}^8 \mu_i M_i$, with parameters $\mu_1,\ldots,\mu_8 \in \fR$ and matrices $M_i$ as in equation \eqref{eq:emmen}, the \emph{homogeneous unfolding} of $L$.
\end{definition}
\begin{remark}
The matrices $M_i$ in the proposition are not unique, since they reflect the choice of a basis in the vector space $\m$. However, this particular choice will turn out to be convenient in later computations.
\end{remark}
\begin{remark}\label{rem:nu5}
Note that $M_5 = L$. When we introduce new parameters by setting $\mu_5 = \nu_0 + \nu_5$ and $\nu_i = \mu_i$ for $i \neq 5$ we obtain $\L(\nu) = (1+\nu_0)L + \H(\nu)$. Thus $\L$ is in fact the unfolding of the family of matrices $(1+\nu_0)L$ provided that $1+\nu_0 \neq 0$. From this we may already infer that the unfolding parameter $\mu_5$ will be relatively unimportant, also see section \ref{sec:stabdombkar}.\commentaar{only works if L is semi simple}
\end{remark}
\begin{remark}
In connection with the previous remark, we may also set $\nu_5 = \mu_5 + 1$ and $\nu_i = \mu_i$ for $i \neq 5$, then $\L(\mu) = \H(\nu)$. The homogeneous unfolding has the advantage that $\H(t \nu) = t \H(\nu)$ for $t \in \fR$. Therefore, since $\Delta$ and $t \Delta$ are equivalent eigenvalue configurations for $t \in \fR_{>0}$, $\H(\nu)$ and $\H(t \nu)$ have equal eigenvalue configurations. However $\H(\nu)$ is an unfolding of $L$ only if $\nu_5 \neq 0$. 
\end{remark}

\subsection{The reduced centralizer unfolding of $L$}\label{sec:unfolredcent}
The aim of this section is to reduce the number of parameters in the centralizer unfolding of $L$. By this we mean that each element of the centralizer unfolding is equivalent to an element of the so called \emph{reduced unfolding} with less parameters. Here equivalence is determined by parameter dependent coordinate changes. As explained in section \ref{sec:methrcu} this is equivalent to taking the quotient of $\m$ with respect to the adjoint action of the Lie group $\M$ corresponding to $\m$. In our specific example $L^* = -L$ and therefore $\M = \Gl_L(4,\fR)$. Because we wish to find the singularities of the boundary of the stability domain of $L$ we take care to avoid quotient singularities. Thus we will take a suitable subgroup of $\M$ such that the quotient space is again a smooth manifold.

Since we reconstruct $\M$ from its Lie algebra $\m$ we have enough information about $\M$ to find all elements with compact orbits. A table of commutators of $\m$ allows us to identify the largest appropriate subgroup of $\M$. The result is given in the next proposition.
\begin{proposition}\label{pro:redunfo1468}
A \emph{reduced centralizer unfolding} is given by $\L_r: \fR^5 \to \gl(4,\fR)$ and
\begin{equation}\label{eq:redunfo1468}
\L_r(\nu) = L + \nu_1 M_1 + \nu_2 M_4 + \nu_3 M_6 + \nu_4 M_8 + \nu_5 M_5.
\end{equation}
\end{proposition}
\begin{remark}
The reduced unfolding is not unique, since it represents an equivalence class. We may also take $L + \nu_1 M_1 + \nu_2 M_4 + \nu_3 M_7 + \nu_4 M_8 + \nu_5 M_5$ for example. The choice made in the proposition is convenient when considering symplectic, reversible and equivariant subsystems, see \cite{hkb}.
\end{remark}

\begin{proof}
In our setting we consider the adjoint action of $\M$ on $\m$
\begin{equation*}
g \cdot A = g^{-1} A g,
\end{equation*}
for $g \in \M$ and $A \in \m$. 

\textbf{Largest subgroup of $\M$ with compact orbits.} Every element $A \in \m$ can be written as 
\begin{equation*}
A = \sum_{i=1}^8 \alpha_i M_i,
\end{equation*}
for certain $\alpha_i \in \fR$. Furthermore every element of $\M$ (in the component continuously connected to the identity) can be written as $\exp(\sum_{i=1}^8 t_i M_i)$, with $t_i \in \fR$. Let us therefore compute $\exp(t M_i)$ for each $i$, then we have $\exp(t M_i) = c_i(t) \id + s_i(t) M_i$, where
\begin{enumerate}\abc
\item $c_1(t) = \exp(t)$ and $s_1(t) = 0$,
\item $c_i(t) = \cosh(t)$ and $s_i(t) = \sinh(t)$ for $i \in \{2,3,4\}$,
\item $c_i(t) = \cos(t)$ and $s_i(t) = \sin(t)$ for $i \in \{5,6,7,8\}$.
\end{enumerate}
Thus $M_5$, $M_6$, $M_7$ and $M_8$ generate compact orbits. However, since $M_5=L$ the adjoint action of $\exp(t M_5)$ on $\m$ is trivial and the same holds for $M_1 = \id$. With help of the table of commutators, see table \ref{tab:comms} we identify the largest subgroup of $\M$ with compact orbits.
\begin{table}[hbtp]
\begin{equation*}
\begin{array}{l|rrr|rrr}
    &  M_2 &  M_3 &  M_4 &  M_6 &  M_7 &  M_8\\\hline
M_2 &   0  & 2M_8 & 2M_7 &    0 & 2M_4 & 2M_3\\
M_3 &-2M_8 &    0 &-2M_6 &-2M_4 &    0 &-2M_2\\
M_4 &-2M_7 & 2M_6 &    0 & 2M_3 &-2M_2 &    0\\\hline
M_6 &    0 & 2M_4 &-2M_3 &    0 &-2M_8 & 2M_7\\
M_7 &-2M_4 &    0 & 2M_2 & 2M_8 &    0 &-2M_6\\
M_8 &-2M_3 & 2M_2 &    0 &-2M_7 & 2M_6 &    0\\
\end{array}
\end{equation*}
\caption{\textit{Commutators $[A,B]$ in $\m$, $A$ vertical and $B$ horizontal. Note that $M_1=\id$ and $M_5=L$ commute with all elements of $\m$.\label{tab:comms}}}
\end{table}

\begin{table}[hbtp]
\begin{equation*}
\begin{array}{l|cccccc}
          & M_2       & M_3       & M_4       & M_6       & M_7       & M_8       \\\hline
\Ad_{tM_6} & M_2       & cM_3-sM_4 & cM_4+sM_3 & M_6       & cM_7+sM_8 & cM_8-sM_7 \\
\Ad_{tM_7} & cM_2+sM_4 & M_3       & cM_4-sM_2 & cM_6-sM_8 & M_7       & cM_8+sM_6 \\
\Ad_{tM_8} & cM_2+sM_3 & cM_3-sM_2 & M_4       & cM_6+sM_7 & cM_7-sM_6 & M_8       \\
\end{array}
\end{equation*}
\caption{\textit{Adjoint action of Lie group $\Mzza$ on $\m$. Notation: $c=\cos(2t)$ and $s=\sin(2t)$.}\label{tab:ads}}
\end{table}

Note that there are several Lie sub-algebras in $\m$, for example $\mzza=\{M=aM_6+bM_7+cM_8\;|\;a,b,c \in \fR\}$ which is similar to $\so(3)$. We denote the corresponding Lie group by $\Mzza$. Clearly $\Mzza$ is the maximal Lie subgroup of $\M$ with compact orbits acting non-trivially on $\m$.

\textbf{Action of $\Mzza$ on $\m$.} Let us compute the adjoint action of this subgroup in table \ref{tab:ads} by using
\begin{equation*}
\Ad_{tM}(A) = \e^{-tM} A \e^{tM}
\end{equation*}
with the generators $\Ad_{tM_6}$, $\Ad_{tM_7}$ and $\Ad_{tM_8}$ of $\Mzza$. Then the adjoint action of $\Mzza$ on $\m$ is an $\SO(3)$-action. Using the fact that every $A \in \m$ can uniquely be written as $A = \sum_{i=1}^8 \alpha_i M_i$ we now obtain an $\SO(3)$-action on $\fR^8$. Again from table \ref{tab:ads} we have
\begin{equation*}
\Ad_{tM_k}(A) = \e^{-tM_k} A \e^{tM_k} = \sum_{i=1}^8 \alpha_i \e^{-tM_k} M_i \e^{tM_k} = \sum_{i=1}^8 \R_k(t,\alpha)_i M_i,
\end{equation*}
where $\R_k$ is defined with help of the rotations $R_1$, $R_2$ and $R_3$ in $\fR^3$ around the $x_1$, $x_2$ and $x_3$ axes. Furthermore when we write $\alpha = (\alpha_1, x, \alpha_5, y) \in \fR \times \fR^3 \times \fR \times \fR^3$ and $x = (\alpha_2,\alpha_3,\alpha_4)$, $y = (\alpha_6,\alpha_7,\alpha_8)$ then
\begin{align*}
\R_6(t,\alpha) &= (\alpha_1, R_1(-2t)\,x,    \alpha_5, R_1(2t)\,y),\\
\R_7(t,\alpha) &= (\alpha_1, R_2(-2t)\,x,    \alpha_5, R_2(2t)\,y),\\
\R_8(t,\alpha) &= (\alpha_1, R_3(\hpm2t)\,x, \alpha_5, R_3(2t)\,y).\\
\end{align*}
The action on $\fR^8$ is generated by $\R_6$, $\R_7$ and $\R_8$. As an $\SO(s)$-action it preserves $||\alpha||$ so we may restrict to the $7$-sphere $||\alpha|| = 1$. Now we see that the action of $\Mzza$ acts trivially in the $\alpha_1$ and $\alpha_5$ directions. Thus for all fixed values of $\alpha_1$ and $\alpha_5$ we have a non-trivial action on the $5$-sphere defined by $\{(x,y) \in \fR^3 \times \fR^3 \;|\; ||x||^2 + ||y||^2 = 1\}$. Considered as an action on $\fR^6 = \fR^3 \times \fR^3$, $\R$ acts as the sum of two isomorphic standard representations of $\SO(3)$ on $\fR^3$.

\textbf{Orbits and orbit types of the $\Mzza$-action.} The orbits of the $\Mzza$-action on $S^5$ are two or three dimensional. Points with three dimensional orbit have trivial isotropy group whereas points with two dimensional orbit have isotropy group $\SO(2)$. To see this we consider the tangent space of the $\Mzza$-orbit of $(x,y)$. It is spanned by the three vectors $(-L_1x, L_1y)$, $(-L_2x, L_2y)$ and $(L_3x, L_3y)$. Where $L_1$, $L_2$ and $L_3$ are the standard generators of $\so(3)$. These vectors are linearly independent provided that $y \neq \lambda x$ for $\lambda \in \fR$, then the tangent space is three dimensional. If $(x,y) = (\alpha z, \beta z)$ with $\alpha^2+\beta^2=1$ and $||z||=1$ then only two vectors are linearly independent, in this case the tangent space is two dimensional.

The action of $\R$ is smooth and proper but it is non-free. Therefore we can not use the general theorem about smooth orbit spaces referred to in section \ref{sec:methrcu}. The points on the $5$-sphere with a non-trivial isotropy group are $\{(\alpha x, \beta x) \in \fR^3 \times \fR^3 \;|\; ||x||^2=1, \alpha^2+\beta^2 = 1\}$ and the isotropy group is $\SO(2)$. This is the only non-trivial isotropy group, so there are two orbit types $\SO(3)$ and $\SO(3)/\SO(2)$. Then it is a result of \cite{ric} and in a somewhat broader context \cite{hud} that the orbit space $\m/\sim\phi$ is a $2$-disc. Here $\phi$ is the action of $\Mzza$ on $S^5$.\commentaar{daar doen we het voorlopig mee}

\textbf{Computing the quotient.} Since the action $\R$ is explicitly given, we can perform explicit computations. For every $(x,y) \in S^5$, that is $||x||^2 + ||y||^2 = 1$, we can find a rotation in $\SO(3)$ taking $(x,y)$ into $(\xi_3 e_3, \tilde{y})$. Here $e_3$ is the third standard basis vector in $\fR^3$, $\xi_3 = ||x||$ and $\tilde{y}$ is the rotated vector $y$. If we now apply $\R_8$, leaving $e_3$ invariant, we can take $(\xi_3 e_3, \tilde{y})$ into for example $(\xi_3 e_3, \eta_1 e_1 + \eta_3 e_3)$ with $\xi_3^2 + \eta_1^2+\eta_3^2 = 1$ and $\eta_1 \geq 0$. Since $\xi_3 \geq 0$ and $\eta_1 \geq 0$ this defines a 2-disc, although with a non-smooth boundary. We do not want a non-smooth quotient space nor do we want a boundary, therefore we take the quotient of $S^5$ with an $\SO(3)/(\fZ_2 \times \fZ_2)$-action so that the orbit space becomes a smooth $2$-sphere, defined by $\{(x,y) \in \fR^3 \times \fR^3 \;|\; \xi_3^2 + \eta_1^2+\eta_3^2 = 1, \xi_1=\xi_2=\eta_2=0\}$.
\end{proof}

\section{Stability domain of the reduced unfolding of $L$}\label{sec:stabdom}
The main goal of this section is to find the stability domain $\stabdom$ in parameter space of the reduced unfolding of $L$. In particular we are interested in the boundary of $\stabdom$. Furthermore we wish to locate several eigenvalue configurations which are relevant for systems with additional structure.

\subsection{Characterization of the boundary of the stability domain}\label{sec:stabdombkar}
The stability domain and in particular its boundary is defined in terms of eigenvalue configurations rather than numerical eigenvalues. On the boundary the corresponding maps have at least one eigenvalue zero or a single pair of purely imaginary eigenvalues. The map $L$ has a double pair of non-zero semi simple imaginary eigenvalues, therefore not all eigenvalue configurations occur on an arbitrary small neighborhood of $L$. Indeed by continuity there can be no zero eigenvalues arbitrary close to $L$. Thus only the eigenvalue configurations $\beta \beta$, $\beta^2$, $\beta_1 \beta_2$ and $\beta\gamma_-$ are relevant for the boundary, see figure \ref{fig:eigcfgs}.

Let us now be more precise about what we mean by \emph{eigenvalue configuration}. It is an equivalence class similar to that in \cite{le82} but defined in a different way.
\begin{definition}\label{def:eigconf}
Two collections of eigenvalues $\Lambda = (\lambda_1,\cdots,\lambda_n)$ and $\Delta = (\delta_1,\cdots,\delta_n)$ belong to the same eigenvalue configuration if a permutation $\pi$ of $\{1,\ldots,n\}$ exists such that if $\pi(i)=j$ then
\begin{enumerate}\itemsep 0pt
\item either $\Sgn (\Re\,\lambda_i) = \Sgn (\Re\,\delta_j)$\\ or $\Re\,\lambda_i = \Re\,\delta_j = 0$, but $\Im\,\lambda_i \neq 0$ and $\Im\,\delta_i \neq 0$\\ or $\lambda_i=\delta_j=0$,
\item the algebraic and geometric multiplicities of $\lambda_i$ are equal to those of $\delta_j$,
\end{enumerate}
for all $i \in \{1,\ldots,n\}$.
\end{definition}
We will use the notation of section \ref{sec:introset} to denote an eigenvalue configuration. For example the eigenvalue configuration $\beta\gamma_-$ denotes the equivalence class $\{(is_1,-is_1,-s_2 + i s_2, -s_2 - i s_2) \;|\; s_1,s_2,s_3 \in \fR_{>0}\}$.

Our starting point here is the homogeneous reduced centralizer unfolding from section \ref{sec:unfolredcent}
\begin{equation*}
\H_r(\nu) = \nu_1 M_1 + \nu_2 M_4 + \nu_3 M_6 + \nu_4 M_8 + \nu_5 M_5.
\end{equation*}
We look for eigenvalue configurations: $\beta\gamma$, $\beta_1 \beta_2$, $\beta \beta$, $\beta^2$. Also see table \ref{tab:eigconfs}. However we must keep in mind that $\H_r$ is a reduced unfolding of $L$ provided that $\nu_5 \neq 0$. The following result characterizes the boundary of the stability domain and moreover shows that $\nu_5$ is unimportant when considering eigenvalue configurations only.
\begin{lemma}\label{lem:eigs}
The eigenvalue configurations of $\H_r(\nu)$ are of the types listed when the parameters $\nu$ satisfy the conditions
\begin{center}
\begin{tabular}{l|l}
Configuration & Conditions on $\nu$\\\hline
$\beta\gamma$ & $(\nu_1^2 - \nu_2^2)(\nu_1^2 + \nu_4^2) + \nu_1^2 \nu_3^2 = 0$, if $\nu_1=\nu_4=0$ then $\nu_2^2 \leq \halfje$\\
$\beta_1 \beta_2$ & $\nu_1=0$, $\nu_2 \nu_4 = 0$\\
$\beta \beta$, $\beta^2$ & $\nu_1=0$, $\nu_4 = 0$, $\nu_2 = \pm \nu_3$\\
\end{tabular}
\end{center}
For all cases we need that $\nu_5 \neq 0$.
\end{lemma}

\begin{proof}
In the proof of this lemma we will use the construction described in section \ref{sec:methstabdomkar}. Here the map $\psi$ parameterizes the hyper-surface of coefficients of polynomials having roots $\alpha \pm i \beta$, $\pm i\gamma$. Such polynomials are of the form $q(t) = ((t-\alpha)^2+\beta^2)(t^2+\gamma^2)$. Thus we have
\begin{equation*}
\psi: \fR^3 \to \fR^5: (\alpha, \beta, \gamma) \mapsto ((\alpha^2 + \beta^2) \gamma^2, -2 \alpha \gamma^2, \alpha^2 + \beta^2 + \gamma^2, -2 \alpha, 1).
\end{equation*}
Eliminating $(\alpha, \beta, \gamma)$ from $\psi(\alpha, \beta, \gamma) = (a_0,a_1,a_2,a_3,a_4)$ we find $f(a) = a_0 a_3^2 + a_4 a_1^2 - a_1 a_2 a_3$. Then for all $t \in \fR_{>0}$ we have $f(ta) = t^3f(a)$ and $f(t^4 a_0, t^3 a_1, t^2 a_2, t a_3, a_4) = t^6 f(a_0,a_1,a_2,a_3,a_4)$. The map $\phi$ maps the unfolding parameters $\nu$ in $\H_r(\nu)$ to coefficients of the eigenvalue polynomial of $\H_r(\nu)$. The explicit expression of $\phi(\nu)$ is rather involved, therefore it will be omitted. Imposing the condition that $\H_r(\nu)$ has eigenvalues of the form $\alpha \pm i \beta$, $\pm i\gamma$ we require that $G(\nu)=0$ where $G = f \circ \phi$. From this last equation we infer that $G$ is homogeneous of degree 6. After some computations we find
\begin{equation}\label{eq:gnu}
G(\nu) = f(\phi(\nu)) = -64 \big[\nu_1^2 + \nu_5^2\big] \big[(\nu_1^2 - \nu_2^2)(\nu_1^2 + \nu_4^2) + \nu_1^2 \nu_3^2\big].
\end{equation}
The first factor in this expression is different from zero because $\nu_5$ is non-zero. This means that the boundary of the stability domain is in fact defined by $F(\nu) = 0$ where $F : \fR^4 \to \fR$ and
\begin{equation}\label{eq:fnu}
F(\nu) = (\nu_1^2 - \nu_2^2)(\nu_1^2 + \nu_4^2) + \nu_1^2 \nu_3^2.
\end{equation}
Note that $F$ is homogeneous of degree 4. A closer inspection reveals that we have to add the inequality $\nu_2^2 \leq \halfje$ if $\nu_1=\nu_4=0$. Moreover the set implicitly defined by $F(\nu)=0$ contains the boundary of the stability domain because we did not impose conditions on $\alpha$ in the eigenvalues $\alpha \pm i \beta$.
\end{proof}

\begin{remark}
In \cite{bot} Bottema looks at singularities of the hyper surface $a_0 a_3^2 + a_1^2 - a_1 a_2 a_3 = 0$ in the space of coefficients of eigenvalue polynomials. This turned out to be sufficient to explain the so called \emph{destabilization paradox}, see \cite{ki07,zie}.
\end{remark}
\begin{remark}
Setting $a_0=1$ in $a_0 a_3^2 + a_1^2 - a_1 a_2 a_3=0$ we get one of the Steiner surfaces: $ a_3^2 + a_1^2 - a_1 a_2 a_3=0$, see \cite{cof}. This surface has two Whitney umbrella singularities.
\end{remark}

From now on we assume that $\nu_5$ has a fixed non-zero value. With some abuse of notation we take $\nu \in \fR^4$. The zero set of the function $F$ in equation \eqref{eq:fnu} defines a hyper surface.
\begin{definition}\label{def:c}
The set $C=\{\nu \in \fR^4\;|\;F(\nu)=0\}$ will be called the \emph{critical set}.
\end{definition}
Since we only required that the eigenvalues of $\H_r(\nu)$ are $\alpha \pm i \beta$, $\pm i\gamma$ without imposing the condition that $\alpha < 0$, the critical set contains the boundary of the stability domain. In the next sections we proceed as follows. First we determine the local and global properties of $C$. Using the function $F$ we find the singularities of $C$. Because of the homogeneity of $F$ they come in straight lines emanating from the origin (in $\fR^4$). This allows us to further reduce the problem to a 3-sphere transversally intersecting the critical set $C$ in what will be called the critical surface $S$. The local properties of $S$ can almost immediately be read off from those of $C$. For the global properties of $S$ we use a Whitney stratification of $S$ which we extend to a Whitney stratification of the 3-sphere. On the strata the eigenvalue configuration is constant. This allows us to identify the stability domain on the 3-sphere and describe the singularities on the boundary. More generally we thus obtain a description of a small neighborhood of $L$ in $\gl(4,\fR)$.

\subsection{Local and global properties of the boundary of the stability domain}\label{sec:stabdomlg}
%

\subsubsection*{Singularities of the critical set $C$}
We start with several properties of $F$ in definition \ref{def:c}. Here we find the location of singular points of $C$. It is more convenient to characterize them as critical points of the critical surface to be defined shortly.
\begin{lemma}\label{lem:fprops}
The polynomial $F$ has the following properties
\begin{enumerate}\abc
\item $F$ is homogeneous: $F(t \nu) = t^4 F(\nu)$ for $t \in \fR$.
\item For $\sigma_i \in \{-1,1\}$, $F(\sigma_1 \nu_1,\ldots,\sigma_4 \nu_4)=F(\nu_1,\ldots,\nu_4)$.
\item Critical points of $F$ in $C$ come at least in straight lines because of the homogeneity. They are $\{(0,0,s,t)\;|\;s,t \in \fR\}$ and $\{(0,s,t,0)\;|\;s,t \in \fR\}$. We furthermore find
\begin{equation*}
\Hess_F(0,0,s,t)=\begin{pmatrix}2(s^2+t^2) & 0 & &\\0 & -2t^2 & &\\&&0&0\\&&0&0\end{pmatrix}
\end{equation*}
These points are transverse self-intersections except when $t=0$. Later on we will relate $t=0$ to an intersection of self-intersections. For the other points we have
\begin{equation*}
\Hess_F(0,s,t,0)=\begin{pmatrix}2(s^2-t^2) & 0 & &\\0 & 0 & &\\&&0&0\\&&0&-2s^2\end{pmatrix}
\end{equation*}
Again we find transverse self-intersections when $t^2 > s^2$. There are degeneracies at $t = \pm s$ that we will relate to Whitney umbrellas.
\end{enumerate}
\end{lemma}

\subsubsection*{Further reduction to critical surface}
To further simplify the analysis we use the fact that $F$ is homogeneous. This strongly suggests to restrict to the intersection of $C$ with the 3-sphere given by $||\nu||=1$. Note that $C$ is transverse to this 3-sphere so that we will not introduce intersection singularities. Thus we consider
\begin{equation*}
\begin{cases}
F(\nu) = (\nu_1^2 - \nu_2^2)(\nu_1^2 + \nu_4^2) + \nu_1^2 \nu_3^2 = 0\\
\nu_1^2 + \nu_2^2 + \nu_3^3 + \nu_4^2 = 1,
\end{cases}
\end{equation*}
which is equivalent with
\begin{equation}\label{eq:crits}
\begin{cases}
G(\nu) = \nu_1^2 - 2 \nu_1^2 \nu_2^2 - \nu_2^2 \nu_4^2 = 0\\
\nu_1^2 + \nu_2^2 + \nu_3^3 + \nu_4^2 = 1
\end{cases}
\end{equation}
By eliminating $\nu_3$ from the equation $F(\nu)=0$ we obtain $G(\nu)=\nu_1^2 - 2 \nu_1^2 \nu_2^2 - \nu_2^2 \nu_4^2 = 0$. We give the following geometric meaning to this manipulation. The 3-sphere $||\nu||=1$ can be considered as two 3-discs $\Dpm=\{(\nu_1,\nu_2,\nu_4) \in \fR^3\;|\; \nu_1^2 + \nu_2^2 + \nu_4^2 \leq 1, \nu_3 = \pm \sqrt{1-(\nu_1^2 + \nu_2^2 + \nu_4^2)} \}$ glued smoothly along their common boundary $\nu_3=0$.
\begin{definition}\label{def:s}
 In the 3-sphere we call $S=\{\nu \in \fR^4 \;|\; G(\nu)=0, ||\nu||=1 \}$ the \emph{critical surface}.
\end{definition}
The next result immediately follows from the homogeneity of $F$.
\begin{corollary}\label{cor:cone}
The set $C$ is a cone over $S$.
\end{corollary}
Therefore we concentrate on the critical surface $S$. It consists of two surfaces $\Spm$, where $\pm$ refers to the sign of $\nu_3$. Here $\Spm$ are the 2-dimensional parts of $\{ \nu \in \Dpm \;|\; G(\nu)=0 \}$. Due to the symmetries of $F$, $\Sm$ is a copy of $\Sp$. The critical surface determines a decomposition of the 3-sphere in 3-dimensional open parts where we have different eigenvalue configurations. The latter is the topic of section \ref{sec:evcnbhdl}. To obtain this decomposition we use the Whitney stratification of the critical surface and the corresponding incidence diagram.

\subsubsection*{Singularities of the critical surface $S$}
The next proposition describes the local structure, in particular the singularities of $S$. The notation $P_1,\ldots,P_4$ and $L_1,\ldots,L_6$ will become clear when we parameterize $\Sp$ and $\Sm$ later on.
\begin{proposition}\label{pro:sings}
The critical surface has the following singularities
\begin{enumerate}\abc
\item Four lines of simple self-intersections: $(\nu_1,\nu_2,\nu_4) = (0,0,t)$ for $t \in [-1,0) \cup (0,1]$ and $(\nu_1,\nu_2,\nu_4) = (0,t,0)$ for $t \in [-\hwt,0) \cup (0,\hwt]$, labelled $L_1,L_2,L_5,L_6$ for $\Sp$ and $L_3,L_4,L_5,L_6$ for $\Sm$,
\item \label{sings:st} Two points of self-tangencies at $(\nu_1,\nu_2,\nu_4) = (0,0,0)$ where the lines of self-intersection meet, labelled $P_5$ and $P_6$ for $\Sp$ and $\Sm$ respectively,
\item \label{sings:wu} Four Whitney umbrella points at $(\nu_1,\nu_2,\nu_4) = (0,\pm\hwt,0)$, labelled $P_1,P_2$ for $\Sp$ and $P_3,P_4$ for $\Sm$.
\end{enumerate}
\end{proposition}

\begin{proof}
The location of the singularities follows from lemma \ref{lem:fprops}. Here we prove the nature of the singularities in parts \ref{sings:st} and \ref{sings:wu}. First we consider $G$, see definition \ref{def:s}, locally at the point $(0,0,0)$: $G(x,y,z) = x^2 - 2 x^2 y^2 - y^2 z^2$. The change of coordinates $\xi = x \sqrt{1-2y^2}$, $\eta = y$ and $\zeta = z$ yields that locally $G(x,y,z)=0$ is equivalent to $\xi^2 - \eta^2 \zeta^2 = 0$ (or even $\xi(\xi-\eta \zeta) = 0$ in yet another set of local coordinates).

Let us next consider $G$ locally at the point $(0,\hwt,0)$ then a local change of coordinates yields the standard form of the Whitney umbrella. Indeed
\begin{equation*}
G(x, \hwt - y, z) = x^2 y (2 \sqrt{2}-2y) - z^2(\halfje-\sqrt{2}y+y^2)
\end{equation*}
then by the change of coordinates
\begin{equation*}
\xi=2x\sqrt{1-\hwt y},\;\;\eta=\hwt y,\;\;\zeta=z\sqrt{\halfje-\sqrt{2}y+y^2}
\end{equation*}
equation $G(x, \hwt - y, z) = 0$ is equivalent to $\xi^2 \eta - \zeta^2 = 0$.
\end{proof}

\subsubsection*{Global properties of the critical surface $S$}
We begin with the observation that the critical surface $S$ as defined by equation \eqref{eq:crits} is a ruled surface. To see this, write the first part of equation \eqref{eq:crits} as $(1-\nu_2^2)\nu_1^2 - \nu_2^2 \nu_4^2 = 0$ and disregard the second part for the moment. Then $S$ is formed by lines through the $\nu_2$-axis parallel to the $\nu_1,\nu_4$-plane. In this respect $S$ is a \emph{Catalan surface} and in particular a \emph{right conoid}. Again from the equation, but more easily from the parameterization below, we infer that $S$ is equivalent to one of the Pl\"ucker family of conoids, namely for $n=1$, where $n$ is the index of the family, see \cite{bg, gas}.

Imposing the additional condition that $||\nu||=1$, we find the following parameterization of $S$.
\begin{proposition}\label{pro:param}
The critical surface $S$ consists of two surfaces $\Spm$ being the images of $[-1,1] \times [0,2\pi]$ in $\Dpm$ for the map $\phipm$. This map is defined as
\begin{equation}\label{eq:phipm}
\phipm: [-1,1] \times [0,2\pi] \to \Spm : (s,t) \mapsto (\hwt\, s\,\cos\,t, \hwt\,\cos\,t, s\,\sin\,t).
\end{equation}
It is one-to-one in most points, the exceptions are $\phipm(s,\frac{\pi}{2}) = \phipm(-s,\frac{3\pi}{2})$ for $s \in [-1,1]$ and $\phipm(0,t) = \phipm(0,2\pi-t)$ for $t \in [0,2\pi]$, where $\phipm$ is two-to-one.
\end{proposition}

\begin{proof}[Proof of proposition \ref{pro:param}] It is easily checked that for all $(s,t) \in [-1,1] \times [0,2\pi]$ we have $G(\phipm(s,t))=0$. Conversely, let $\nu$ satisfy $G(\nu)=0$. Fix $\nu_2=c$ then $G(\nu)=0$ reduces to $(1-2c^2)\nu_1^2 - c^2 \nu_4^2=0$ which clearly has pairs of straight lines as solutions as long as $1 - 2 c^2 > 0$ and $c > 0$ (a ruled surface). For fixed $t$, $\phipm(s,t)$ maps onto these straight lines. If $c=0$ or $2c^2=1$ the solutions are single straight lines which we find for $t=\halfje \pi$ and $s=0$ respectively. Thus $\phipm$ maps onto the 2-dimensional part of $S$. The properties $\phipm(s,\frac{\pi}{2}) = \phipm(-s,\frac{3\pi}{2})$ and $\phipm(0,t) = \phipm(0,2\pi-t)$ are again easily checked.
\end{proof}

Note that the points where the map $\phipm$ fails to be one-to-one correspond to curves of self-intersection of $S$. See figure \ref{fig:domain} for the domains of the maps $\phipm$. With help of the previous proposition we find the global structure of $S$. To this end we use a Whitney stratification of $S$. For sake of completeness we state a definition, see \cite{arn2}.
\begin{definition}
Let $V$ as a smooth subset of a manifold $M$ be a topological space. Furthermore let a collection of subspaces $U_i \subset V$, called strata, be given for $i$ in an index set $I$. Then $\{U_i\}_I$ is called a \emph{Whitney stratification} if 1) $\{U_i\}_I$ is a stratification and 2) for each pair $X \subset \cl{Y}$ and for all $y_k \in Y$ with $y_k \to x \in X$, the tangent space $T = \lim_{k \to \infty} T_{y_k} Y$ satisfies $T_x X \subset T$.
\end{definition}
From the differentiability properties of the maps $\phipm$ in equation \eqref{eq:phipm} we almost immediately get the following corollary.
\begin{corollary}\label{cor:incid}
The collection $\{P_1,\ldots,P_6,L_1,\ldots,L_6,S_1,\ldots,S_4\}$ forms a Whitney stratification of $S$. The organization of the stratification is shown in the incidence diagram in figure \ref{fig:incid}.
\end{corollary}

There are several ways to describe the global structure of the critical set $S$. We start with the most 'natural' one which leads us to a more useful description for identifying the stability domain on the $3$-sphere.
\begin{lemma}\label{lem:globnat}
The open parts $S_1$ and $S_2$ form a topological 2-sphere and similarly $S_3$ and $S_4$ form a topological 2-sphere. These two 2-spheres transversely intersect along the circle $P_5 \cup L_5 \cup P_6 \cup L_6$ except at the points $P_5$ and $P_6$. On each 2-sphere there are two (singular) crease intervals $P_1 \cup L_1 \cup P_5 \cup L_2 \cup P_2$ and $P_3 \cup L_3 \cup P_6 \cup L_4 \cup P_4$. At these intervals the 2-spheres touch. The intersecting 2-spheres $\cl{S}_1 \cup \cl{S}_2$ and $\cl{S}_3 \cup \cl{S}_4$ decompose the 3-sphere in the disjoint 3-dimensional open parts $V_1,\dots,V_4$.
\end{lemma}
From the last property a more useful description in terms of $V_1,\dots,V_4$ follows. The latter are domains in the 3-sphere with the same eigenvalue configuration.
\begin{lemma}\label{lem:globuse}
The open parts $S_1$ and $S_4$ form a topological 2-sphere $\cl{S}_1 \cup \cl{S}_4$ enclosing the open topological 3-disc $V_1$, similarly for the pairs $(S_1,S_3)$, $(S_3,S_2)$ and $(S_2,S_4)$ enclosing $V_2$, $V_3$ and $V_4$.
\end{lemma}

\begin{proof}[Proof of corollary \ref{cor:incid} and lemmas \ref{lem:globnat}, \ref{lem:globuse}] The incidence diagram can almost immediately be read off from the domains of $\phipm$ in figure \ref{fig:incid}. For the sake of simplicity we use the same names for the strata in the domains as their images. Since $\phipm$ is one-to-one in most points this should not lead to confusion. Points where $\phipm$ is not one-to-one are indicated in figure \ref{fig:domain}.

From the incidence diagram we infer that the points $P_1,\ldots,P_4$ and the lines $L_1,\ldots,L_4$ are not essential for the global structure of $S$. Therefore the lines indicating their relation with other strata are dashed. Indeed, since $P_1 \cup L_1 \cup P_5 \cup L_2 \cup P_2$ is a smooth interval in the domain of $\phip$, so is its image. Thus $\phip(P_1 \cup L_1 \cup P_5 \cup L_2 \cup P_2)$ is a smooth curve with endpoints $P_1$ and $P_2$ (or to be formally correct $\phip(P_1)$ and $\phip(P_2)$). This means we may shrink the lengths of $L_1$ and $L_2$ to zero without changing the global structure. Similarly for $P_3 \cup L_3 \cup P_6 \cup L_4 \cup P_4$.

Now we turn to $\cl{S}_1 \cup \cl{S}_2$. From the domain of $\phip$ we see that $S_1$ and $S_2$ restricted to $\Dp$ are smoothly attached along the lines $L_5$ and $L_6$. Therefore they form a 2-disc in $\Dp$. The boundary of $S_1$ and $S_2$ restricted to $\Dp$ is a great circle, namely the image of the dashed line with four arrows in figure \ref{fig:domain}. This great circle is shown in figure \ref{fig:conoid}. Similarly the restriction of $S_1$ and $S_2$ to $\Dm$ is also a 2-disc with the same boundary. These 2-discs are glued smoothly along their common boundary thus forming a 2-sphere. It is only a topological 2-sphere because there is a singular interval on each 2-disc. The same holds for $S_3$ and $S_4$.

Looking again at the domains of $\phipm$ in figure \ref{fig:domain} we infer that $P_5 \cup L_5 \cup P_6 \cup L_6$ is a smooth closed curve in $S$. In fact it is in $S$ the boundary of each $S_i$. Therefore $\cl{S}_1 \cup \cl{S}_4$ for example is a topological 2-sphere. The intersections of $S_1$, $S_2$, $S_3$ and $S_4$ with the boundary of $\Dp$, and of $\Dm$, consist of two transversally intersecting great circles. Thus we see for example that $S_2$ and $S_3$ are both on the same side, let us agree that they are on the \emph{outside}, of the topological $2$-sphere $\cl{S}_1 \cup \cl{S}_4$. Then the inside is an open $3$-disc which we call $V_1$. A similar construction holds for the pairs $(S_1,S_3)$, $(S_3,S_2)$ and $(S_2,S_4)$, yielding the open 3-discs $V_2$, $V_3$ and $V_4$.
\end{proof}

\section{Eigenvalue configurations on a neighborhood of $L$}\label{sec:evcnbhdl}
We are now in a position to describe the eigenvalue configurations on a small neighborhood of $L$ in $\gl(4,\fR)$. In order to do so we apply a number of reductions. First we restrict to a neighborhood in the unfolding of $L$. By construction of the centralizer unfolding every matrix $A$ near $L$ is equivalent to some member of the centralizer unfolding $\L(\mu)$ of $L$, with $\mu \in \fR^8$. Then, using the reduced unfolding, every element of $\L(\mu)$ is equivalent to some member of the reduced centralizer unfolding $\L_r(\nu)$ of $L$, now with $\nu \in \fR^5$. The third step consists of switching to the homogeneous reduced unfolding $\H_r(\nu)$. The results of the previous section show that for $\nu_5 \neq 0$, a condition which is fulfilled in a small neighborhood of $L$, the eigenvalue configurations are constant along rays $t \nu$, with $\nu \in \fR^4$ and $t \in (0,\rho]$. Here $\rho > 0$ depends on $\nu_5$. In view of remark \ref{rem:nu5} we may assume that $\nu_5$ is large enough so that we may take $\rho = 1$. Then we use homogeneity of eigenvalue configurations to restrict to the $3$-sphere $\{\nu \in \fR^4 \;|\; ||\nu|| = 1\}$. Thus we left with a three dimensional problem.

For the latter we have a Whitney stratification $\{P_1,\ldots,P_6,L_1,\ldots,L_6,S_1,\ldots,S_4,V_1,\ldots,V_4\}$ with increasing dimension, see figure \ref{fig:incid}. On each of these strata the eigenvalue configuration is constant. By choosing representative points and computing the eigenvalues we get table \ref{tab:eigconfs}. From this table we infer that the stability domain of $\H_r(\nu)$ is $V_3$. The boundary of the stability domain can be read off from the incidence diagram in figure \ref{fig:incid}, it consists of the components $\{S_2,S_3,L_1,\ldots,L_6,P_1,\ldots,P_6\}$. That is, the topological $2$-discs $S_2$ and $S_3$ and moreover all singular lines and points, see figure \ref{fig:conoid}.

All that remains is to prove that table \ref{tab:eigconfs} is correct.

\begin{proof}[Proof table \ref{tab:eigconfs} is correct] The eigenvalue configurations follow from direct computation. The only thing left to prove is the difference between the nilpotent and the semi-simple case. The matrix $\L_r(\nu)$ has a double pair of imaginary eigenvalues if $\nu_1=\nu_4=0$ and $\nu_3=\pm\nu_2$. Then $\L_r(0,\nu_2,s\nu_2,0) = L + \nu_2 (M_4 + sM_6)$ with $s=\pm 1$. By construction of the unfolding $[L,M_4+sM_6]=0$ and using equation (\ref{eq:emmen}) we see that $(M_4 + sM_6)^2=0$. Thus if $S=L$ and $N=\nu_2 (M_4 + sM_6)$ then $S$ is semi-simple, $N$ is nilpotent, $N^2=0$ and $[S,N]=0$, so $S+N$ is indeed the Jordan-Chevalley decomposition of $\L_r(0,\nu_2,s\nu_2,0)$. A double pair of imaginary eigenvalues only occurs at the points $P_1,\ldots,P_4$ and $\L_r$ is not semi-simple.
\end{proof}

\section*{Acknowledgement}
The research of O.N.K. was supported by DFG Grant No. HA 1060/43-1.


\end{document}